\documentclass[12pt]{amsart}
\usepackage{amssymb}
\usepackage{amsmath}
\usepackage{amsthm}
\usepackage{color}
\usepackage{hyperref,url}
\usepackage{ mathrsfs }
\usepackage{graphicx}
\usepackage{slashed}
\usepackage{enumitem}
\usepackage{tikz}
\usetikzlibrary{patterns,arrows.meta}
\usepackage[normalem]{ulem}

\hypersetup{
	colorlinks=true,
	linkcolor=blue,
	filecolor=blue,
	urlcolor=black,
	citecolor=blue,
}

\definecolor{plum}{rgb}{.5,0,1}

\theoremstyle{plain}
\newtheorem{theorem}{Theorem}
\newtheorem{lemma}{Lemma}
\newtheorem{proposition}{Proposition}

\newtheorem{remark}{Remark}

\newtheorem{definition}{Definition}

\newcommand{\R}{\mathbb{R}}
\newcommand{\C}{\mathbb{C}}

\newcommand{\D}{\mathbb{D}}
\newcommand{\B}{\mathbb{B}}
\renewcommand{\P}{\mathbb{P}}

\newcommand{\T}{\mathbb{T}}
\newcommand{\e}{\epsilon}

\newcommand{\supp}{\mathrm{supp}}

\newcommand{\bth}{\boldsymbol{\theta}}
\newcommand{\bT}{\mathbf{T}}

\newcommand{\cN}{\mathcal{N}}

\newcommand{\I}{\mathcal{I}}

\newcommand{\mc}{\mathcal}

\def\XXint#1#2#3{{\setbox0=\hbox{$#1{#2#3}{\int}$ }
		\vcenter{\hbox{$#2#3$ }}\kern-.6\wd0}}

\newcommand{\stkout}[1]{\ifmmode\text{\sout{\ensuremath{#1}}}\else\sout{#1}\fi}

\usepackage{bbm}
\usepackage{scalerel,stackengine}
\stackMath
\newcommand\widecheck[1]{%
	\savestack{\tmpbox}{\stretchto{%
			\scaleto{%
				\scalerel*[\widthof{\ensuremath{#1}}]{\kern-.6pt\bigwedge\kern-.6pt}%
				{\rule[-\textheight/2]{1ex}{\textheight}}
			}{\textheight}%
		}{0.5ex}}%
	\stackon[1pt]{#1}{\scalebox{-1}{\tmpbox}}%
}
\newcommand{\Dec}{\textup{Dec}}
\newcommand{\BilDec}{\textup{BilDec}}
\begin{document}
	
	\title{Decoupling for AD-regular sets on the parabola}

	\author{Ciprian Demeter} \address{ Ciprian Demeter\\  Department of Mathematics\\ Indiana University Bloomington, USA} \email{demeterc@iu.edu}
	
	\author{Yuqiu Fu} \address{Yuqiu Fu}\email{yuqiufu25@gmail.com}
	
	\thanks{The first author is  partially supported by the NSF grant DMS-2349828}
	
	\begin{abstract}
		We improve the decoupling exponent for functions with spectrum inside AD-regular collections of arcs on the parabola. We achieve this by incorporating recent Szemer\'{e}di--Trotter-type estimates into the bootstrapping argument from \cite{BD}. As an application, our results complement, and in some cases improve earlier results \cite{chang2022decoupling} for arithmetic Cantor sets.  
	\end{abstract}

	\maketitle

	\section{Introduction}
	Let us start with a definition.
	\begin{definition}
		Let $\alpha\in(0,1]$.
		
		For $R\ge 1$, let $\I_R$ be a collection of pairwise disjoint intervals of length $R^{-1}$ inside $[-1,1]$, and let $C_{AD}\ge 1$. We call $\I_R$ an $(\alpha,R,C_{AD})$-AD regular collection if for each interval $J$ of length $2/R\le |J|\le 4$ centered at some $x\in\cup_{I\in\I_R}I$ we have
		\begin{equation}
			\label{kdjchpvuryue8urtg9t08h=670980}
			C_{AD}^{-1}(|J|R)^{\alpha}\le \#\{I\in\I_R:\;I\subset J\}\le C_{AD}(|J|R)^{\alpha}.
		\end{equation} 
	\end{definition}
	When $\I_R$ is the whole partition of $[-1,1]$, we may take $\alpha=1$ and $C_{AD}\sim 1$. Examples for $\alpha<1$ are included in Section \ref{sec:2}. Note that
	$$\#(\I_R)\sim_{C_{AD}}R^{\alpha}.$$

	Consider the truncated parabola
	$$\P=\{(x,x^2):\;x\in[-1,1]\}.$$
	We write $\cN_{\delta}$ for its $\delta$-neighborhood. For each collection $\I$ of intervals, we consider the corresponding subset of this neighborhood 
	$$\cN_{\delta}(\I)=\{(x,x^2+t):\;x\in\cup_{I\in\I}I,\;|t|\le \delta\}.$$
	We  denote by $\Theta(\I_{\delta^{-1/2}})$ the collection of rectangle-like, essentially flat regions $\theta$ with dimensions $\delta^{1/2}\times \delta$, that cover $\cN_{\delta}(\I_{\delta^{-1/2}})$. We use notation
	$$\widehat{F_\theta}=\widehat{F}1_\theta.$$
	\begin{theorem}
		\label{mmmain1}
		Assume $\alpha\in(0,1)$. For each $6<p<\infty$ there is $c_{p,\alpha}>0$ such that for each  $(\alpha,R^{1/2},C_{AD})$-AD regular collection $\I=\I_{R^{1/2}}$ and each  $F:\R^2\to\C$ with spectrum inside $\cN_{1/R}(\I)$ we have 
		\begin{equation}
			\label{ ruehfr0eufui rt9go5tho0-uoj=-}
			\|F\|_{L^p(\R^2)}\lesssim_{\epsilon,C_{AD},p,\alpha} R^{\epsilon} (\#\Theta(\I))^{\frac12-\frac3p-c_{p,\alpha}}(\sum_{\theta\in \Theta(\I)}\|F_\theta\|^2_{L^p(\R^2)})^{1/2}.
		\end{equation}
	\end{theorem}
	Implicit constants  will typically depend on $p$ and $\alpha$, but we will rarely record that. For $R$-dependent quantities $A_R,B_R$ we will use $A_R\lessapprox B_R$ to denote the fact that the inequality 
	$A_R\lesssim_\epsilon R^\epsilon B_R$ holds true for arbitrary $\epsilon>0$. For example,
	$$(O(\log R))^{O((\log\log R)^{O(1)})}\lessapprox 1.$$
	The dependence on $C_{AD}$ in \eqref{ ruehfr0eufui rt9go5tho0-uoj=-} may be tracked to be of the form $(C_{AD})^{O(\log\log R)^{O(1)}}$. This is indeed $\lessapprox_{C_{AD}}1.$   
	
	The exponent $c_{p,\alpha}$ we get is easily computable, though it looks less appealing. 
	There is an interesting question about what the best possible value of $c_{p,\alpha}$ should be. Taking $F=F_\theta$ shows that 
	$$c_{p,\alpha}\le \frac12-\frac3p.$$
	Constructive interference for $F$ equal to the exponential sum with one frequency for each $I$  shows that
	$$
	c_{p,\alpha}\le \frac{3(1-\alpha)}{p\alpha}.
	$$
	This upper bound is applicable to arbitrary collections $\I$, it does not take advantage of regularity. In Section \ref{sec:2} we show that within the AD-regular framework, the following stronger constraint arises from energy considerations
	\begin{equation}
		\label{cehf 7ut9ugh9ythytih96i0j=}
		c_{p,\alpha}\le \frac{2(1-\alpha)}{p\alpha}.
	\end{equation}
	It seems tempting to conjecture that the optimal $c_{p,\alpha}$ is 
	\begin{equation}
		\label{dbvh bdfhvkfkhsgh jw;gh}
		c_{p,\alpha}=\min(\frac12-\frac3p,\frac{2(1-\alpha)}{p\alpha}).
	\end{equation}
	This is equivalent to upgrading \eqref{ ruehfr0eufui rt9go5tho0-uoj=-} to square root cancellation 
	\begin{equation}
		\label{mfjviu9uirt09ihi0-9j7=}
		\|F\|_{L^p(\R^2)}\lesssim_{\epsilon,C_{AD}} R^{\epsilon} (\sum_{\theta\in \Theta(\I)}\|F_\theta\|^2_{L^p(\R^2)})^{1/2}
	\end{equation}
	for $$p=2+\frac4\alpha.$$

	This is only known to be true when $\alpha=1$.  The inequality 
	\begin{equation}
		\label{cdh urhreoifyierfyoiy}
		\|F\|_{L^p(\R^2)}\lessapprox (\#\Theta(\I))^{\frac12-\frac3p}(\sum_{\theta\in \Theta(\I)}\|F_\theta\|^2_{L^p(\R^2)})^{1/2}
	\end{equation}
	was proved in \cite{BD} for $p\ge 6$, for arbitrary collections $\I$ of pairwise disjoint $R^{-1/2}$-intervals. In the absence of any regularity assumption, and apart from $R^\epsilon$-losses, the exponent $\frac12-\frac3p$ is sharp for certain collections of (any number of) intervals. For example, this is the case whenever the centers of the intervals form an arithmetic progression, as can be seen from constructive interference, by testing \eqref{cdh urhreoifyierfyoiy} with exponential sums. Testing the same example in Theorem \ref{mmmain1} shows that at most $${M}\lessapprox_{C_{AD}} (\frac1d)^{\frac{3}{3+pc_{p,\alpha}}}$$ many of the intervals $I\in\I$ can lie inside an arithmetic progression with step $d$, for each $(\alpha,R^{1/2},C_{AD})$-regular collection $\I$. However, it is very easy to see the bound that follows from  \eqref{kdjchpvuryue8urtg9t08h=670980}, namely
	$$M\lesssim C_{AD}^2(\frac1d)^{\alpha}.$$
	Due to \eqref {cehf 7ut9ugh9ythytih96i0j=} this latter bound is better than the previous one, since $\frac{3}{3+pc_{p,\alpha}}> \alpha$. Thus, not surprisingly, even a  decoupling with a sharp $c_{p,\alpha}$ would not have interesting consequences for the additive structure of AD-regular sets on $\R$. Instead, our decoupling implies certain lack of additive structure of AD-regular sets on the parabola. See Remark \ref{adden}. There are a growing number of papers addressing this phenomenon. We refer the reader to Question 2.13 in \cite{BD}, as well as \cite{O1}, \cite{DD}, \cite{O2}, \cite{DemWang}, \cite{Orpflat}, \cite{doprod}  and the references therein.

	To prove Theorem \ref{mmmain1} we use the original bootstrapping argument in \cite{BD}, that combines $L^2$ orthogonality with bilinear $(p,\frac{p}2)$ decoupling. The latter is a reformulation of the bilinear restriction inequality. The reason for using  the half exponent $p/2$ comes from the role of the $(4,2)$ decoupling in the bilinear setting, and interpolation with the other endpoint $(\infty,\infty)$.

	Our  new critical input here is an improved bilinear $(p,\frac{p}2)$ decoupling. 
	
	Let $J_1,J_2\subset [-1,1]$ be intervals separated by $\sim 1$. For $i\in\{1,2\}$, let $\I_i$ be $(\alpha,R^{1/2},C_{AD})$-AD regular collections of intervals inside $J_i$.
	\begin{theorem}[Improved bilinear $(p,\frac{p}2)$ decoupling]
		\label{mmmain2}
		Let $\alpha\in(0,1)$ and let $4<p<\infty$.
		Assume  $F_i:\R^2\to\C$ has spectrum inside $\cN_{1/R}(\I_i)$. There is $$\Gamma_{p,\alpha}<\alpha(\frac14-\frac1p)-\frac2p$$
		such that
		\begin{equation}
			\label{ ruehfr0eufui rt9go5tho0-uoj=-fghytjh}
			\|(F_1F_2)^{1/2}\|_{L^p(\R^2)}\lesssim_{\epsilon,C_{AD}} R^{\Gamma_{p,\alpha}+\epsilon} \prod_{i=1}^2(\sum_{\theta\in \Theta(\I_i)}\|F_{i,\theta}\|^2_{L^{p/2}(\R^2)})^{1/4}.
		\end{equation}
	\end{theorem}
	See \eqref{jh fh rufhu0 rey98g9} for an explicit  relation between the exponents in Theorems \ref{mmmain1} and \ref{mmmain2}.

	Inequality \eqref{ ruehfr0eufui rt9go5tho0-uoj=-fghytjh} is known to hold with exponent $\alpha(\frac14-\frac1p)-\frac2p$ in place of $\Gamma_{p,\alpha}$ for arbitrary collections $\I_i$ of $M=R^{\alpha/2}$ many intervals, without any regularity assumption. This follows by interpolating $p=4$ with $p=\infty$, see Section \ref{sec:5}. One sharp example comes by considering $M$ overlapping wave packets for each $i$, in consecutive directions. Constructive interference can be enforced on a $\sqrt{R}/M$-square. 
	The better exponent $\Gamma_{p,\alpha}$ in \eqref{ ruehfr0eufui rt9go5tho0-uoj=-fghytjh} is a consequence of AD-regularity, as consecutive directions are now  forbidden. 
	
	In particular, $\Gamma_{4,\alpha}=-\frac12$ for each $\alpha\in(0,1]$. It is an interesting question as to what the sharp exponent $\Gamma_{p,\alpha}$ should be for $p>4$. The most severe lower bound we found is
	\begin{equation}
		\label{;l jreguti g]9ry]0 p5o0h679=}
		\Gamma_{p,\alpha}\ge \begin{cases}-\frac2p,\;p\le \frac4\alpha\\\frac\alpha4-\frac3p,\;p>\frac4\alpha\end{cases}.
	\end{equation}
	We do not exclude the possibility that more intricate constructions could improve this lower bound. \eqref{;l jreguti g]9ry]0 p5o0h679=}  follows from constructive interference as follows. Assume each $F_{i,\theta}$ is a single, non-oscillatory wave packet 
	$$F_{i,\theta}\approx 1_{T_\theta},$$
	with $T_{\theta}$ a $\sqrt{R}\times R$ rectangle centered at the origin. Let $M\lesssim R^{1/2}$.  Due to regularity, we can only have $M^\alpha$ out of $M$ consecutive directions contributing to  each $F_i$. In the saturated case,   $|F_i|\sim M^{\alpha}$ on a $\sqrt{R}/M\times R/{M^2}$ rectangle $R_i$ centered at the origin. Due to transversality, $|R_1\cap R_2|\sim R/M^2$, and thus \eqref{ ruehfr0eufui rt9go5tho0-uoj=-fghytjh} implies that
	$$M^{\alpha}(R/M^2)^{1/p}\lesssim R^{\Gamma_{p,\alpha}}M^{\alpha/2}(R^{3/2})^{2/p}.$$
	To get \eqref{;l jreguti g]9ry]0 p5o0h679=} we choose $M=1$ when $p\le 4/\alpha$ and $M=\sqrt{R}$ when $p>4/\alpha$.
	
	It would be particularly interesting to know whether $\Gamma_{p,\alpha}=-\frac2p$ for some $p>4$. If true, this would entail square root cancellation. Indeed, the square function version of this inequality holds true, see \eqref{inv6}. It would also imply \eqref{mfjviu9uirt09ihi0-9j7=}, as explained in Remark  \ref{ckl jhuicrhii rt[pgi]0-h96}. 
	\begin{remark}
		\label{chyurt8g0-5696 j9- boy0h u-0j670j= 5u0-}
		It is not clear whether the half exponent pair $(p,p/2)$ is as fundamental when $\alpha<1$, as it is in the case $\alpha=1$. Other $(p,q)$ decouplings might be more relevant to consider in order to improve the bootstrapping mechanism when $\alpha<1$. See Remark \ref{weh -fureigrtihyj u=jpi8=[ko []}.
	\end{remark}
	\medskip
	
	Theorem \ref{mmmain2} is a consequence of Theorem \ref{mmmain3} and interpolation, see Section \ref{sec:5}.
	\begin{theorem}[Improved bilinear $(8,4)$ decoupling]
		\label{mmmain3}
		Let $\alpha\in(0,1)$.
		Assume  $F_i:\R^2\to\C$ has spectrum inside $\cN_{1/R}(\I_i)$. There is $$\Gamma_{8,\alpha}<\frac{\alpha}{8}-\frac14$$
		such that
		\begin{equation}
			\label{ ruehfr0eufui rt9go5tho0-uoj=-fghytjhipf urogut-gu}
			\|(F_1F_2)^{1/2}\|_{L^8(\R^2)}\lesssim_{\epsilon,C_{AD}} R^{\Gamma_{8,\alpha}+\epsilon} \prod_{i=1}^2(\sum_{\theta\in \Theta(\I_i)}\|F_{i,\theta}\|^2_{L^{4}(\R^2)})^{1/4}.
		\end{equation}
	\end{theorem}
	The proof will reveal that we may take
	\[
	\Gamma_{8,\alpha}+\frac14=
	\begin{cases}\displaystyle \frac{\alpha}{10},& \displaystyle \alpha\le \frac12,\\
		\displaystyle \frac{\alpha}{4(3-\alpha)},& \displaystyle \frac12<\alpha\le \frac23,\\[6pt]
		\displaystyle \frac{-3\alpha^2+13\alpha-2}{32(3-\alpha)},& \displaystyle \frac23\le \alpha<1.
	\end{cases}
	\]

	There is nothing special about $p=8$, other than potentially being the critical exponent for $\Gamma_{p,\alpha}$ when $\alpha=\frac12$, assuming \eqref{;l jreguti g]9ry]0 p5o0h679=} is sharp. We chose to work with this concrete exponent in order to simplify the computations. We could  instead have worked out our argument for generic $p$, with possibly a different outcome than the one we get by interpolation. See also Remark \ref{kjhe foreygeug-98y9670=u90=679} on the sharpness of our approach. 
	\medskip

	At the heart of our argument for Theorem \ref{mmmain3} when $\alpha\le 1/2$ lies the Szemer\'{e}di--Trotter-type bound in \cite{DemWang},  for rectangles (tubes) satisfying certain non-concentration assumptions. That result was proved (and is only true in that form) in the case when  $\alpha\le \frac12$. The sharp analogous result for $\alpha\in(\frac12,1)$ was proved more recently  in \cite{doprod}. 
	
	The inspiration for our approach to Theorem \ref{mmmain3} comes from an argument used by Guth to reprove decoupling for the parabola. Its details are presented in Section 10.4 of \cite{demeter2020fourier}.
	It relies on the bilinear Kakeya inequality. One of the difficult points in our argument is to find the right way to incorporate the stronger  Szemer\'{e}di--Trotter-type estimate. It seems possible that a variant of this argument could be used to prove Theorem \ref{mmmain1} directly, without using Theorem \ref{mmmain2}.
	\medskip
	
	The proof of Theorem \ref{mmmain3} in Sections \ref{sec:4} and \ref{sec:large-alpha-computation} takes advantage of all the scales. Strictly speaking, it does not use induction on scales, it relies merely on a direct argument based on efficiently iterating inequalities derived in Section \ref{sec:3}. On the other hand, induction on scales is used in Section \ref{sec:5} as in \cite{BD}, to prove Theorem \ref{mmmain1} based on Theorem \ref{mmmain2}.    
	\subsection*{Acknowledgements}
	
	The authors thank OpenAI’s GPT-5.5 Pro, 
	accessed through Indiana University’s institutional ChatGPT Edu service, for assistance with the bookkeeping computations in Section \ref{sec:4} and Section \ref{sec:large-alpha-computation}. The authors  initially established \eqref{cklh hfghrug i[rtuh[h9]]} for $K=1$ themselves, using only $L_0$ and $L_1$. This corresponds to $\Gamma_{8,\alpha}=\frac{5\alpha}{48}-\frac14$ in Theorem \ref{mmmain3} (when $\alpha\le \frac12$). A small further improvement was then obtained by the authors, by also using $L_2$. That version of the argument was then successfully fed to the AI, with the task of finding the optimal exponent.

	All mathematical claims and final verification remain the responsibility of the authors.
	
	The second author would like to thank Larry Guth for many helpful discussions.
	
	The first author is grateful to Zane Li for motivating conversations and suggestions.

	\medskip
	
	\section{An application to Cantor sets}
	\label{sec:2}
	Let us now describe one of the motivations behind our Theorem \ref{mmmain1}, following the work \cite{chang2022decoupling} on generalized arithmetic Cantor sets. Given $n\ge k\ge 2$ and an alphabet $\D=\{d_1, d_2, \ldots, d_k\} \subset \{0, 1, \ldots, n - 1\},$ we write
	\begin{equation}\label{2}
		C=\mathcal{C}_{n}^{\{d_1,\ldots, d_k\}} = \left\{ \sum_{j=1}^{\infty} a_j n^{-j}: a_j \in  \D\right\}.
	\end{equation}
	This is a self-similar set with Hausdorff dimension 
	$$\dim(C)=\alpha=\frac{\log k}{\log n}.$$
	For a given $C = \mathcal{C}_{n}^{d_1,\ldots, d_k},$ and $i\ge 1$ we write 
	$$C_i =  \left\{ \sum_{j=1}^i a_j n^{-j}: a_j \in  \D\right\}$$ 
	and let $\I_i$ be the
	collection of $k^i$ many pairwise disjoint intervals $I$ of the form 
	\begin{equation}
		\label{xjkhugryfugu49[i569]}
		[\;\sum_{j=1}^ia_jn^{-j}, \sum_{j=1}^ia_jn^{-j}+n^{-i}\;],\;a_j\in\{d_1,\ldots,d_k\}.
	\end{equation}
	The collection $\I_i$  is $(\alpha,n^i,C_{AD})$-AD regular, with
	$C_{AD}\sim k.$ Indeed, the upper bound in \eqref{kdjchpvuryue8urtg9t08h=670980} is always satisfied with $C_{AD}\sim k^{1-\alpha}$, as the worst case scenario is a fully packed interval of length $k/n^i$. The lower bound in \eqref{kdjchpvuryue8urtg9t08h=670980} is always satisfied with $C_{AD}\sim k$, as the worst case scenario is having an interval of length $ \frac12n^{-i+1}$ containing only one point in $C_i$.

	We let $D_p(C_i)$ denote the best constant such that
	\begin{equation}
		\label{dycg yuef7reyf854-t856y9}
		\|\sum_{I \in \mc{I}_i} F_I\|_{L^p(\R)} \leq D_p(C_i) (\sum_{I \in \mc{I}_i} \|F_I\|^2_{L^p(\R)} )^{1/2} 
	\end{equation}
	holds for every $F_I: \R\rightarrow \C$ with $\supp \hat F_I \subset I.$ An interpolation argument between $L^2$ and $L^\infty$ shows that
	\begin{equation}
		\label{;kljreiofgy4r849t9-580}
		D_p(C_i)\lesssim k^{i(\frac12-\frac1p)},\;2\le p<\infty,
	\end{equation}
	with an implicit constant independent of $k$, $i$, $n$ and the alphabet. This inequality is sometimes far from being sharp. Indeed, due to self-similarity we see that (for some $A_p\ge 1$)
	$$D_{p}(C_{i+j})\le A_pD_{p}(C_{i})D_{p}(C_j),$$
	in particular $$D_{p}(C_i)\le (A_pD_p(C_1))^{i-1}.$$
	Fix a large enough constant $B_p$, $p>2$.
	Choose a $\Lambda(p)$ alphabet $\D$ with implicit constant at most $B_p$,
	$$D_p(C_1)\le B_p,$$
	for $p$ satisfying $n^{2/p}\sim k$. See \cite{Lamp} for the existence of such $\D$. 
	In this case
	$$D_p(C_i)\le (A_pB_p)^{i}.$$
	However,  $k,n$ may be chosen so that $k^{\frac12-\frac1p}$ is as large as we wish, much larger than $A_pB_p$.
	
	There are also instances when \eqref{;kljreiofgy4r849t9-580} is essentially sharp. We present two types of examples. The first one relies on constructive interference, and becomes increasingly more efficient when $\alpha$ is close to 1. We test the inequality \eqref{dycg yuef7reyf854-t856y9} with 
	$$f(x)=\phi(\frac{x}{n^i})\sum_{f\in C_i} e(fx)$$
	where $\phi$ is a smooth approximation of $1_{[-1,1]}$ with spectrum inside $[-1,1]$. Since
	$$|f(x)|\gtrsim k^i, \;\;|x|\ll 1$$
	we find that for each $p\ge 2$ 
	$$D_p(C_i)\gtrsim \frac{k^{i(\frac12-\frac1p)}}{(\frac{n}k)^{\frac{i}{p}}}.$$
	This lower bound is close to the upper bound \eqref{;kljreiofgy4r849t9-580} when $(n/k)^{1/p}\sim 1$.

	For the second example we use an even $p$, and the fact that 
	$$D_p(C_i)\gtrsim \frac{(E_p(C_i))^{1/p}}{(\#C_i)^{1/2}},$$
	where the $p$-energy is defined as 
	$$E_p(C_i)=\#\{(f_1,\ldots,f_{p})\in (C_i)^{p}:\;f_1+\ldots+f_{p/2}=f_{p/2+1}+\ldots+f_{p}\}.$$
	This follows by noting that if $\widehat{\phi^p}(0)=1$ and $\widehat{\phi}\ge 0$, then
	\begin{align*}
		\|\phi(\frac{x}{n^i})\sum_{f\in C_i} e(fx)\|_{p}^p&=n^i\sum_{f_1,\ldots,f_p\in C_i}\widehat{\phi^p}({n^i}(f_1+\ldots+f_{p/2}-f_{p/2+1}-\ldots-f_p))\\&\ge n^{i}E_p(C_i).
	\end{align*}
	If we choose the alphabet to be an arithmetic progression of length $k$, then  an easy exercise shows that, for some appropriate $K_p$
	$$E_p(C_i)\ge (\frac{k^{p-1}}{K_p})^{i}.$$
	This shows that for each $i$
	\begin{equation}
		\label{chrfy8g=-569y0967=78=78-}
		D_p(C_i)\gtrsim (\frac{k^{\frac12-\frac1p}}{K_p^{1/p}})^{i}.
	\end{equation}
	This lower bound is close to the upper bound \eqref{;kljreiofgy4r849t9-580} when, say, $\log k>K_p^{1/p}$.
	
	\bigskip
	
	Now consider the partition $\Omega_i$ of the $n^{-2i}$-neighborhood $\mathcal{N}_{n^{-2i}}(\I_i)$
	into sets $\theta_I$, where $\theta_I = \mathcal{N}_{n^{-2i}} \cap (I \times \R)$ for $I \in \mathcal{I}_i.$ We note that each $\theta_I$ is roughly a $n^{-i} \times n^{-2i}$ rectangle, and there are $k^i$ many $\theta_I.$
	We let $\Dec_p(C_i)$ denote the best constant such that
	\begin{equation} 
		\label{c h vhryg7yt0gt08-56856=97-6}
		\|\sum_{I \in \mc{I}_i} F_{\theta_I}\|_{L^p(\R^2)} \leq \Dec_p(C_i) (\sum_{I \in \mc{I}_i} \|F_{\theta_I}\|^2_{L^p(\R^2)} )^{1/2}
	\end{equation}
	holds for every $F_{\theta_I}:\R^2 \rightarrow \C$ with $\supp(\hat F_{\theta_I}) \subset \theta_I.$ 
	\medskip
	
	\cite{chang2022decoupling} proved the following relationship between $D_p(C_i)$ and $\Dec_p(C_i).$ We caution that $N(1)$ and $\dim(C)$ in \cite{chang2022decoupling} are denoted here by $k$ and $\alpha$, respectively.
	
	\begin{theorem}\label{thm1}
		Let $p\ge 2$.	Let $C$ be an arithmetic Cantor set defined as in \eqref{2}, and let $C_i$, $D_p(C_i),$ $\Dec_p(C_i)$ be defined as in the previous paragraphs. 
		Assume that 
		\begin{equation}
			\label{hhuihreufy8gu8tug89uh8-578}
			D_p(C_i) \lesssim k^{ \kappa_p i}
		\end{equation}
		for some $\kappa_p\le \frac12-\frac1p$ and each $i$. Then
		\begin{equation}\label{3vjirgrtgurtihu}
			\Dec_{3p}(C_i) \lesssim_{\e,k,\alpha} k^{(\kappa_p+\epsilon) i}
		\end{equation}
		for each $i$ and $\e>0$.
	\end{theorem}   
	We note that $\lesssim_{\e,k,\alpha}$ is the same as $\lesssim_{\e,k,n}$. 
	
	At its core, the proof \cite{BD} of the decoupling for the parabola ($\alpha=1$) relies on a bootstrapping argument that combines bilinear Kakeya with $L^2$ orthogonality. The derivation of \eqref{3vjirgrtgurtihu} can be seen from the same argument, with \eqref{hhuihreufy8gu8tug89uh8-578} used in place of $L^2$ orthogonality. See Remark \ref{finrem}.
	
	Inequality \eqref{3vjirgrtgurtihu} becomes powerful when \eqref{hhuihreufy8gu8tug89uh8-578} is known to hold 
	with $\kappa_p<\frac12-\frac1p$, as the latter represents an improvement over $L^2$ orthogonality.
	
	Theorem \ref{thm1}  gives (via \eqref{;kljreiofgy4r849t9-580}) the universal estimate 
	\begin{equation}
		\label{ewjurefe8g9u56-8g-56y-67u=78=}
		\Dec_{p}(C_i)\lesssim_{\epsilon,k,\alpha} (\# C_i)^{\frac12-\frac3p+\epsilon},\;p\ge 6
	\end{equation}
	for arbitrary $C_i$. However, in reality this bound does not depend on self-similarity or AD-regularity. It simply follows by interpolating the $L^6$ decoupling for the parabola
	\cite{BD} with the trivial $L^\infty$ bound, that only relies on the cardinality of $C_i$. 
	
	We note that when $\log k>K_p^{1/p}$, the lower bound  \eqref{chrfy8g=-569y0967=78=78-}
	shows that \eqref{hhuihreufy8gu8tug89uh8-578}
	does not hold for any $\kappa_p<\frac12-\frac1p-\log_k(K_p^{1/p})$, when $C_i$ as in the second example from above. When applied  to these sets, Theorem \ref{thm1} delivers an upper bound no better than 
	\begin{align*}
		\Dec_{3p}(C_i)&\lesssim_{\e,k,\alpha}k^{(\frac12-\frac1p-\log_k(K_p^{1/p})+\epsilon)i}\\&\sim_{\e,k,\alpha,p}k^{(\frac12-\frac1p+\epsilon)i},\;p\ge 2.
	\end{align*}
	This is the same as the generic upper bound \eqref{ewjurefe8g9u56-8g-56y-67u=78=}.

	As a consequence of our Theorem \ref{mmmain1}, we improve on \eqref{ewjurefe8g9u56-8g-56y-67u=78=} in all cases. We also improve on \eqref{3vjirgrtgurtihu} whenever
	$$\frac12-\frac3p-c_{p,\alpha}<\kappa_{p/3}.$$
	
	\begin{theorem}\label{mainthm}
		Let $C$ be an arithmetic Cantor set defined as in \eqref{2}. For each $6<p<\infty$, let $c_{p,\alpha}>0$ be the exponent in Theorem \ref{mmmain1}. Then for each $i\ge 1$
		\begin{equation}\label{rjhef refurtgu09h3}
			\Dec_{p}(C_i) \lesssim_{\e,k,\alpha}  k^{i(\frac{1}{2} - \frac{3}{p} - c_{p,\alpha}+\e)}.
		\end{equation}
	\end{theorem}
	\begin{proof}
		Write $R=n^{2i}$. The collection $\I_i$ in  \eqref{xjkhugryfugu49[i569]} is $(\alpha,R^{1/2},C_{AD})$-AD regular, with 
		$C_{AD}\sim k.$
		
		We apply Theorem \ref{mmmain1} to find 
		$$\Dec_{p}(C_i)\lesssim_{\e,k,\alpha}R^{\e}k^{i(\frac12-\frac3p-c_{p,\alpha})}.$$
		Since $R^\e=k^{i2\e/\alpha}$, we are done.
	\end{proof}
	\begin{remark}
		\label{adden}
		Consider the lattice points on the parabola given by
		$$S_i=\{(\sum_{j=0}^{i-1}a_jn^j,(\sum_{j=0}^{i-1}a_jn^j)^2):\;a_j\in\D\}.$$
		An application of Cauchy--Schwarz shows that for each even $p\ge 2$
		$$E_p(S_i)\gtrsim \frac{(\#S_i)^p}{\#(S_i+\ldots+S_i)},$$
		where the sum in the denominator is $p/2$-fold. 
		
		Assume now the alphabet $\D=\{d_1,\ldots,d_k\}$ is an arithmetic progression, so $C_i$ is a generalized arithmetic progression of dimension $i$. Since $$S_i+\ldots+S_i\subset\{\sum_{j=0}^{i-1}b_jn^j:\;b_j\in \D+\ldots+\D\}\times \{0,1,\ldots,pn^{2i}/2\},$$
		and since
		\begin{equation}
			\label{bb lfdvjhjsgjkh s;gruwgiuo}
			\#(\D+\ldots+\D)\sim_p\#\D,
		\end{equation}
		we find that
		$$\#(S_i+\ldots+S_i)\lesssim_p(O_p(1)k)^in^{2i},$$
		and thus
		$$E_p(S_i)\gtrsim_p\frac{k^{i(p-1)}}{(O_p(1)n)^{2i}}.$$
		The standard argument of testing \eqref{c h vhryg7yt0gt08-56856=97-6} with exponential sums
		shows that for each even $p\ge 2$
		\begin{equation}
			\label{cbd vlfdghv flskh}
			E_p(S_i)\lesssim \Dec_{p}(C_i)^p(\#S_i)^{p/2}.
		\end{equation}
		These together force the lower bound
		$$\Dec_{p}(C_i)\gtrsim \left(\frac{k^{\frac12-\frac1p}}{O_p(1)n^{2/p}}\right)^{i}=\left(\frac{k^{\frac12-\frac1p-\frac2{p\alpha}}}{O_p(1)}\right)^{i}.$$
		Comparing this with \eqref{rjhef refurtgu09h3} when $p\ge 6$, we find that
		$$c_{p,\alpha}\le \frac{2(1-\alpha)}{p\alpha}.$$
		Theorem \ref{mmmain1} together with \eqref{cbd vlfdghv flskh} shows that for each even $p\ge 8$
		$$E_p(S_i)\lesssim_{\epsilon,k} n^{i\e}(\#S_i)^{p-3-pc_{p,\alpha}}.$$
		If \eqref{dbvh bdfhvkfkhsgh jw;gh} is indeed true, this would imply diagonal behavior 
		$$E_p(S_i)\lesssim_{\epsilon,k} n^{i\e}(\#S_i)^{p/2}$$
		for the critical exponent $p=2+\frac4\alpha$. This in turn implies that if the $p/2$-fold sumset $S_i+\ldots+S_i$ has minimal expansion \eqref{bb lfdvjhjsgjkh s;gruwgiuo} in the $x$-direction, it should have essentially maximal expansion in the $y$-direction, filling in a significant fraction of $\{0,1,\ldots,pn^{2i}/2\}$. 
	\end{remark}

	\section{Preliminaries for the proof of Theorem \ref{mmmain3}}
	\label{sec:3}
	In this section we perform a multi-scale decomposition, introduce a number of key parameters, and prove key structural inequalities.  
	\medskip
	
	Fix two intervals $J_1,J_2\subset [-1,1]$ separated by  $\sim 1$. For $i\in\{1,2\}$, let $\I_i$ be an $(\alpha,R^{1/2},C_{AD})$-AD regular collection of intervals inside $J_i$. For $j\ge 0$ we call $\I_{i,j}$ the $(\alpha,R^{2^{-j-1}},C_{AD})$-AD regular collection of $1/R^{2^{-j-1}}$-intervals covering $\I_i$. Note that $\I_{i,0}=\I_i$.

	Given $R$, we pick the smallest integer $M$ such that $R^{\frac{1}{2^M}}\le 2$. Note that $$M\sim \log\log R.$$  
	
	Throughout the argument, the index $j$ will range through $\{0,1,\ldots,M\}$, and $i\in\{1,2\}$.
	We consider the radii $R_j=R^{\frac1{2^j}}$ and denote by $B_{R_j}$ a generic ball with radius $R_j$. 
	We write $\Theta_{i,j}$ for the collection of $R_j^{-1/2}\times R_j^{-1}$-rectangles $\theta_j$ that cover $\cN_{1/R_j}(\I_{i,j})$.

	We will use $$\#\{\theta_j\in\Theta_{i,j}:\;\theta_j\subset \theta_{j+1}\}\le C_{AD}^2R_{j+1}^{\alpha/2}\lesssim_{C_{AD}}R_{j+1}^{\alpha/2}.$$
	Due to its irrelevance, the dependence on $C_{AD}$ will not be recorded explicitly in this section; we will simply write $\lesssim$ rather than $\lesssim_{C_{AD}}$. Each inequality will have  an implicit loss of at most  $(C_{AD})^{O(1)}$. In the next section, when these inequalities will be combined, the  losses will accumulate to at most $(C_{AD})^{O(M)}$, which is $\lessapprox_{C_{AD}}1$. 
	
	We write $T_j$ for a generic $R_j^{1/2}\times R_j$-rectangle dual to some $\theta_j$, lying inside $B_{R_j}$.
	
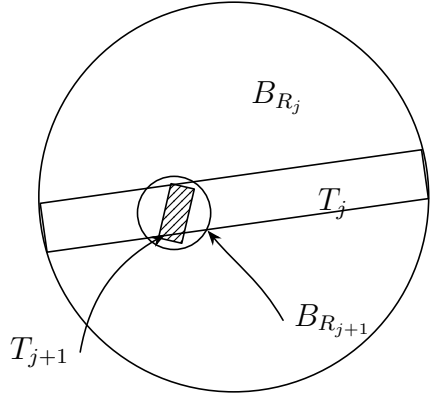
	\begin{figure}[ht]
		\centering
		\begin{tikzpicture}[scale=1.05,line cap=round,line join=round,>=Stealth]
			\draw[semithick] (0,0) circle (2.45);
			\begin{scope}[rotate=8]
				\draw[semithick] (-2.42,-0.36) rectangle (2.42,0.26);
			\end{scope}
			\node at (0.55,1.25) {$B_{R_j}$};
			\node at (1.25,-0.08) {$T_j$};
			\draw[semithick] (-0.75,-0.20) circle (0.46);
			\begin{scope}[shift={(-0.75,-0.20)},rotate=-13]
				\draw[semithick,pattern=north east lines] (-0.12,-0.35) rectangle (0.18,0.35);
			\end{scope}
			\draw[semithick,->] (-1.93,-1.95) to[out=80,in=210] (-0.86,-0.47);
			\node[anchor=east] at (-1.93,-1.95) {$T_{j+1}$};
			\draw[semithick,->] (0.62,-1.55) to[out=120,in=310] (-0.33,-0.42);
			\node[anchor=west] at (0.62,-1.55) {$B_{R_{j+1}}$};
		\end{tikzpicture}
	\end{figure}
	Fix a ball $B_R=B_{R_0}$ of radius $R$. Fix two functions $F_1,F_2$ with spectrum inside $\cN_{1/R}(\I_i)$ 
	
	For a cleaner exposition, we will use $\approx$ and $\lessapprox$ somewhat casually, hiding standard technicalities under the rug.
	\\
	\\
	Step 0. Pigeonholing at level $j=0$. The parameters $H_{i,0}$, $M_{i,0}$, $\bth_{i,0}$. 
	\\
	\\
	Recall that
	$$F_i(x)=\sum_{\theta_0\in\Theta_{i,0}}F_{i,\theta_0}(x).$$
	We use the standard wave packet decomposition (see e.g.  \cite{demeter2020fourier})
	$$F_{i,\theta_0}(x)\approx\sum_{T\in\T(\theta_0,B_R)}F_{i,T}(x),\;\;x\in B_R$$
	where $\T(\theta_0,B_R)$ is the collection of tubes dual to $\theta_0$, intersecting $B_R$.
	\medskip
	
	Since we aim to prove Theorem \ref{mmmain3}, we seek an upper bound for the quantity
	$$\frac{\int_{B_R}|F_1F_2|^4}{(\sum_{\theta_0\in\Theta_{1,0}}\|F_{1,\theta_0}\|^2_{L^4(B_R)})^2(\sum_{\theta_0\in\Theta_{2,0}}\|F_{2,\theta_0}\|^2_{L^4(B_R)})^2}.$$
	
	First, we select parameters $H_{i,0}$, $M_{i,0}$ such that there is a collection $\Theta_{i,0}(B_R)\subset\Theta_{i,0}$ and a collection $\T'(\theta_0,B_R)\subset\T(\theta_0,B_R)$ for each $\theta_0\in \Theta_{i,0}(B_R)$  such that 
	$$\#(\T'(\theta_0,B_R))\sim M_{i,0}$$
	$$\|F_{i,T_0}\|_{\infty}\sim H_{i,0},\;\forall \;T_0\in\T'(\theta_0,B_R)$$
	$$\int_{B_R}|F_1F_2|^4\approx \int_{B_R}|F_1'F_2'|^4$$
	where 
	$$F_i'=\sum_{T_0\in \T'_{i,0}(B_R)}F_{i,T_0},$$
	$$\T'_{i,0}(B_R)=\cup_{\theta_0\in\Theta_{i,0}(B_R)}\T'(\theta_0,B_R).$$
	The existence of such collections follows from the triangle inequality, and the fact that there are $\lessapprox 1$ many relevant values for the parameters $H_{i,0}$ and $M_{i,0}$. Writing 
	$$F_{i,\theta_0}'=\sum_{T_0\in \T'(\theta_0,B_R)}F_{i,T_0}$$
	we note that 
	\begin{equation}
		\label{jufgyyg8i 0-i=239r058t90}\frac{\int_{B_R}|F_1F_2|^4}{(\sum_{\theta_0\in\Theta_{1,0}}\|F_{1,\theta_0}\|^2_{L^4(B_R)})^2(\sum_{\theta_0\in\Theta_{2,0}}\|F_{2,\theta_0}\|^2_{L^4(B_R)})^2}\lessapprox \lambda_{-1}(B_R),
	\end{equation}
	where
	$$
	\lambda_{-1}(B_R)=\frac{\int_{B_R}|F'_1F'_2|^4}{(\sum_{\theta_0\in\Theta_{1,0}(B_R)}\|F'_{1,\theta_0}\|^2_{L^4(B_R)})^2(\sum_{\theta_0\in\Theta_{2,0}(B_R)}\|F'_{2,\theta_0}\|^2_{L^4(B_R)})^2}.
	$$
	This is due to the fact that (see e.g. Exercise 2.7 (d) in \cite{demeter2020fourier})
	$$\|F'_{i,\theta_0}\|_{L^4(B_R)}\lesssim \|F_{i,\theta_0}\|_{L^4(B_R)}.$$
	We write
	$$\bth_{i,0}=\#\Theta_{i,0}(B_R).$$
	\\
	\\
	Step 1. Pigeonholing at level $j=1$. Parameters $r_{i,1}$, $r_{i,0\subset1}$, $H_{i,1}$, $M_{i,1}$, $\bth_{i,1}$. The collection $\B_0$ and the distinguished ball $B_{R_1}^*$.
	\\
	\\
	Consider a finitely overlapping cover of $B_{R}$ by balls $B_{R_1}$. We first note that
	$$F_i'(x)\approx\sum_{T_0:\;T_0\cap B_{R_1}\not=\emptyset}F_{i,T_0}(x),\;x\in B_{R_1}.$$
	We select a subcollection $\B_0$ of these balls satisfying three types of conditions.
	
	First, we require that
	\begin{equation}
		\label{jufgyyg8i 0-i=239r058t901}
		\int_{B_R}|F'_1F'_2|^4\approx \sum_{B_{R_1}\in\B_0}\int_{B_{R_1}}|F'_1F'_2|^4.
	\end{equation}
	
	Second, we assume the existence of two constants denoted by $r_{i,{1}}$ and $r_{i,0\subset 1}$ such that (for each $i\in\{1,2\}$), each $B_{R_1}\in\B_0$ intersects a special collection called $\T_{i,B_{R_1}\supset}$, consisting of
	$$\sim N_{i,0}=r_{i,{1}}r_{i,0\subset 1}$$
	many $T_0\in \T_{i,0}'(B_R)$. There is of course (essentially) a unique $T_0\in \T_{i,B_{R_1}\supset}$ for each contributing direction $\theta_0$.  Moreover, we require that for each $i$, these directions are organized as follows. There are $\sim r_{i,1}$ many $\theta_1\in\Theta_{i,1}$, each containing $\sim r_{i,0\subset 1}$ many contributing directions $\theta_0$. 
	Write
	$$F'_{i,B_{R_1}}=\sum_{T_0\in \T_{i,B_{R_1}\supset}}F_{i,T_0}'.$$
	Finally, we 
	require that 
	\begin{equation}
		\label{jhfureguriegi09hi[y]}
		\int_{B_{R_1}}|F_1'F_2'|^4\approx \int_{B_{R_1}}|F'_{1,B_{R_1}}F'_{2,B_{R_1}}|^4.
	\end{equation}
	For any given $B_{R_1}$, the existence of parameters $r_{i,1}$, $r_{i,0\subset 1}$ with this property follows from the triangle inequality and pigeonholing. Since the range of the relevant such parameters has size $\lessapprox 1$, we can choose $\B_0$ to be represented by the same parameters. 
	The function $F_{i,B_{R_1}}'$ has spectrum inside $\cN_{1/R}$, thus also inside $\cN_{1/R_1}$. Thus, it has a wave packet decomposition on $B_{R_1}$ 
	$$F_{i,B_{R_1}}'(x)\approx \sum_{\theta_1\in\Theta_{i,1}}\sum_{T_1\in \T(\theta_1,B_{R_1})}F'_{i,B_{R_1},T_1}(x),\;x\in B_{R_1}.$$
	The first sum is restricted to  $r_{i,1}$ many $\theta_1$. For each such $\theta_1$, there are $\sim r_{i,0\subset 1}$ many $\theta_0\subset \theta_1$, each contributing with a tube $T_0(\theta_0)\in \T_{i,B_{R_1}\supset}$. We have the following connection between wave packet decompositions at the two scales
	\begin{equation}
		\label{krjghrtughptriguitruhuh9ui}  
		\sum_{T_1\in \T(\theta_1,B_{R_1})}F'_{i,B_{R_1},T_1}(x)\approx \sum_{\theta_0\subset \theta_1}F'_{i,T_0(\theta_0)}(x),\;x\in B_{R_1}.
	\end{equation}
	
	Third, we can make the selection of $\B_0$ so that there are  parameters $H_{i,1}$, $M_{i,1}$, $\bth_{i,1}$, a collection $\Theta_{i,1}(B_{R_1})\subset\Theta_{i,1}$ and a collection $\T'(\theta_1,B_{R_1})\subset\T(\theta_1,B_{R_1})$ for each $\theta_1\in \Theta_{i,1}(B_{R_1})$  such that 
	$$\#(\Theta_{i,1}(B_{R_1}))\sim\bth_{i,1},\;\forall\;B_{R_1}\in\B_0$$
	$$\#(\T'(\theta_1,B_{R_1}))\sim M_{i,1},\;\forall\;\theta_1\in \Theta_{i,1}(B_{R_1})$$
	$$\|F'_{i,B_{R_1},T_1}\|_{\infty}\sim H_{i,1},\;\forall \;T_1\in\T'(\theta_1,B_{R_1})$$
	\begin{equation}
		\label{jufgyyg8i 0-i=239r058t902}
		\int_{B_{R_1}}|F_1'F_2'|^4\approx \int_{B_{R_1}}|F_1''F_2''|^4
	\end{equation}
	where 
	\begin{equation}
		\label{krjghrtughptriguitruhuh9ui1}  
		F_i''(x)= \sum_{T_1\in \T'_{i,1}(B_{R_1})}F'_{i,B_{R_1},T_1}(x),\;x\in B_{R_1}
	\end{equation}
	$$\T'_{i,1}(B_{R_1})=\cup_{\theta_1\in\Theta_{i,1}(B_{R_1})}\T'(\theta_1,B_{R_1}).$$

	We of course do not require that the direction set $\Theta_{i,1}(B_{R_1})$ is the same for each $B_{R_1}$, only that its size is independent of $B_{R_1}$.
	
	We end this round of pigeonholing by collecting some useful information. Comparing \eqref{krjghrtughptriguitruhuh9ui} and \eqref{krjghrtughptriguitruhuh9ui1}, and using the $L^2$ orthogonality (on any $B_{R_1}\in \B_0$) of the  wave packets corresponding to different $T_1$, we have
	\begin{align*}
		\|\sum_{T_1\in \T'_{i,1}(B_{R_1})}F'_{i,B_{R_1},T_1}\|_{L^2(B_{R_1})}&\lesssim \|\sum_{T_1\in \T(\theta_1,B_{R_1})}F'_{i,B_{R_1},T_1}\|_{L^2(B_{R_1})}\\&\approx \|\sum_{\theta_0\subset \theta_1}F'_{i,T_0(\theta_0)}\|_{L^2(B_{R_1})}.
	\end{align*}
	Both first and third terms are computable, as the wave packets associated with distinct $T_0$ are also orthogonal on $B_{R_1}$. The above inequality becomes
	\begin{equation}
		\label{n re hvhturtgu56i96i0-67}
		H_{i,1}\bth_{i,1}^{1/2}M_{i,1}^{1/2}R_1^{3/4}\lessapprox H_{i,0}r_{i,0\subset 1}^{1/2}R_1.
	\end{equation}
	Next, pick $\theta_1\in \Theta_{i,1}(B_{R_1})$. Using \eqref{krjghrtughptriguitruhuh9ui}, the (essential) pairwise disjointness of the spatial supports of $F'_{i,B_{R_1},T_1}$ and the triangle inequality we get
	\begin{equation}
		\label{n re hvhturtgu56i96i0-671}
		H_{i,1}\lessapprox \|\sum_{T_1\in \T(\theta_1,B_{R_1})}F'_{i,B_{R_1},T_1}\|_{\infty}\approx \|\sum_{\theta_0\subset \theta_1}F'_{i,T_0(\theta_0)}\|_{\infty}\lesssim H_{i,0}r_{i,0\subset 1}.
	\end{equation}

	Write
	$$
	\lambda_0(B_{R_1})= \frac{\int_{B_{R_1}}|F''_1F''_2|^4}{(\sum_{\theta_1\in\Theta_{1,1}(B_{R_1})}\|F''_{1,\theta_1}\|^2_{L^4(B_{R_1})})^2(\sum_{\theta_1\in\Theta_{2,1}(B_{R_1})}\|F''_{2,\theta_1}\|^2_{L^4(B_{R_1})})^2}.
	$$
	Note that due to \eqref{jufgyyg8i 0-i=239r058t901}, \eqref{jhfureguriegi09hi[y]} and \eqref{jufgyyg8i 0-i=239r058t902} we have
	$$\lambda_{-1}(B_R)\lessapprox \max_{B_{R_1}\in\B_0}\lambda_0(B_{R_1})\times$$$$\frac{\sum_{B_{R_1}\in\B_0}{(\sum_{\theta_1\in\Theta_{1,1}(B_{R_1})}\|F''_{1,\theta_1}\|^2_{L^4(B_{R_1})})^2(\sum_{\theta_1\in\Theta_{2,1}(B_{R_1})}\|F''_{2,\theta_1}\|^2_{L^4(B_{R_1})})^2}}{(\sum_{\theta_0\in\Theta_{1,0}(B_R)}\|F'_{1,\theta_0}\|^2_{L^4(B_R)})^2(\sum_{\theta_0\in\Theta_{2,0}(B_R)}\|F'_{2,\theta_0}\|^2_{L^4(B_R)})^2}.$$We choose $B_{R_1}^{*}\in\B_0$ to maximize $\lambda_{0}(B_{R_1})$.
	\\
	\\
	Step 2. Pigeonholing at level $j=2$. Parameters $r_{i,2}$, $r_{i,1\subset2}$, $H_{i,2}$, $M_{i,2}$, $\bth_{i,2}$. The collection $\B_1$ and the distinguished ball $B_{R_2}^*$.
	\\
	\\
	We repeat the selection process we applied to $B_{R}$ in the previous step. This time we do it for $B_{R_1}^{*}$. Here are the details.
	
	Consider a finitely overlapping cover of $B_{R_1}^{*}$ by balls $B_{R_2}$.
	We select a subcollection $\B_1$ of these balls as follows.
	
	First, we require 
	\begin{equation}
		\label{djc uhruiyruifyypg}
		\int_{B_{R_1}^{*}}|F''_1F''_2|^4\approx \sum_{B_{R_2}\in\B_1}\int_{B_{R_2}}|F''_1F''_2|^4.
	\end{equation}
	
	Second, we assume the existence of two constants denoted by $r_{i,{2}}$ and $r_{i,1\subset 2}$ such that (for each $i\in\{1,2\}$), each $B_{R_2}\in\B_1$ intersects a special collection $\T_{i,B_{R_2}\supset}$
	consisting of
	$$\sim N_{i,1}=r_{i,{2}}r_{i,1\subset 2}$$
	many $T_1\in \T_{i,1}'(B_{R_1})$. Moreover, we require that for each $i$, the directions $\theta_1$ of these $T_1$ are organized as follows. There are $\sim r_{i,2}$ many $\theta_2\in\Theta_{i,2}$, each containing $\sim r_{i,1\subset 2}$ many contributing directions $\theta_1$. Write
	$$F''_{i,B_{R_2}}=\sum_{T_1\in \T_{i,B_{R_2}\supset}}F_{i,T_1}''.$$
	We 
	require that 
	\begin{equation}
		\label{jhfureguriegi09hi[y]bfgbyt hytj uj}
		\int_{B_{R_2}}|F_1''F_2''|^4\approx \int_{B_{R_2}}|F''_{1,B_{R_2}}F''_{2,B_{R_2}}|^4.
	\end{equation}
	
	Third, we can make the selection of $\B_1$ so that there are  parameters $H_{i,2}$, $M_{i,2}$, $\bth_{i,2}$, a collection $\Theta_{i,2}(B_{R_2})\subset\Theta_{i,2}$ and a collection $\T'(\theta_2,B_{R_2})\subset\T(\theta_2,B_{R_2})$ for each $\theta_2\in \Theta_{i,2}(B_{R_2})$  such that 
	$$\#(\Theta_{i,2}(B_{R_2}))\sim\bth_{i,2},\;\forall\;B_{R_2}\in\B_1$$
	$$\#(\T'(\theta_2,B_{R_2}))\sim M_{i,2},\;\forall\;\theta_2\in \Theta_{i,2}(B_{R_2})$$
	$$\|F'_{i,B_{R_2},T_2}\|_{\infty}\sim H_{i,2},\;\forall \;T_2\in\T'(\theta_2,B_{R_2})$$
	\begin{equation}
		\label{djc uhruiyruifyypg1}
		\int_{B_{R_2}}|F_1''F_2''|^4\approx \int_{B_{R_2}}|F_1'''F_2'''|^4
	\end{equation}
	where 
	$$F_i'''(x)= \sum_{T_2\in \T'_{i,2}(B_{R_2})}F''_{i,B_{R_2},T_2}(x),\;x\in B_{R_2}$$
	$$\T'_{i,2}(B_{R_2})=\cup_{\theta_2\in\Theta_{i,2}(B_{R_2})}\T'(\theta_2,B_{R_2}).$$
	As in the previous round of pigeonholing, we get
	$$H_{i,2}\bth_{i,2}^{1/2}M_{i,2}^{1/2}R_2^{3/4}\lessapprox H_{i,1}r_{i,1\subset 2}^{1/2}R_2$$
	via $L^2$ orthogonality, and we get 
	$$H_{i,2}\lessapprox H_{i,1}r_{i,1\subset 2}$$
	via the triangle inequality.

	Write
	$$\lambda_1(B_{R_2})= \frac{\int_{B_{R_2}}|F'''_1F'''_2|^4}{(\sum_{\theta_2\in\Theta_{1,2}(B_{R_2})}\|F'''_{1,\theta_2}\|^2_{L^4(B_{R_2})})^2(\sum_{\theta_2\in\Theta_{2,2}(B_{R_2})}\|F'''_{2,\theta_2}\|^2_{L^4(B_{R_2})})^2}.$$
	Note that due to \eqref{djc uhruiyruifyypg}, \eqref{jhfureguriegi09hi[y]bfgbyt hytj uj} and \eqref{djc uhruiyruifyypg1}
	we have
	$$\lambda_{0}(B_{R_1}^{*})\lessapprox \max_{B_{R_2}\in\B_1}\lambda_1(B_{R_2})\times$$$$\frac{\sum_{B_{R_2}\in\B_1}{(\sum_{\theta_2\in\Theta_{1,2}(B_{R_2})}\|F'''_{1,\theta_2}\|^2_{L^4(B_{R_2})})^2(\sum_{\theta_2\in\Theta_{2,2}(B_{R_2})}\|F'''_{2,\theta_2}\|^2_{L^4(B_{R_2})})^2}}{(\sum_{\theta_1\in\Theta_{1,1}(B_{R_1}^{*})}\|F''_{1,\theta_1}\|^2_{L^4(B_{R_1^{*}})})^2(\sum_{\theta_1\in\Theta_{2,1}(B_{R_1}^{*})}\|F''_{2,\theta_1}\|^2_{L^4(B_{R_1}^{*})})^2}.$$We end Step 2 by choosing $B_{R_2}^{*}\in\B_1$, to maximize $\lambda_{1}(B_{R_2})$.
	\\
	\\
	Summary of pigeonholing: We repeat this process and construct a nested sequence of balls
	$$B_{R_M}^*\subset B_{R_{M-1}}^*\subset\ldots\subset B_{R_1}^*\subset B_{R_0}^*=B_{R}$$
	and collections $\B_j$ consisting of balls $B_{R_{j+1}}$ inside $B_{R_j}^*$. 
	
	We also construct the functions $F_i^{(j)}$ iteratively as follows. First, $F_i^{0}=F_i$. Assuming $F_i^{(j+1)}$ was constructed in Step $j$ as a sum of wave packets 
	\begin{equation}
		\label{rhughtgytygygu9gu90}
		F_{i}^{(j+1)}(x)\approx  \sum_{T_j\in\T'_{i,j}(B_{R_j}^*)}F_{i,T_j}^{(j)}(x),\;x\in B_{R_j^*},
	\end{equation}
	we first construct $F_{i,B_{R_{j+1}}}^{(j+1)}$ as the part of $F_{i}^{(j+1)}$ that only retains the wave packets $F_{i,T_j}^{(j)}$ according to parameters $r_{i,j+1}$, $r_{i,j\subset j+1}$.
	
	We then expand $F_{i,B_{R_{j+1}}}^{(j+1)}$ into wave packets $F_{i,B_{R_{j+1}},T_{j+1}}$ of smaller scale on each $B_{R_{j+1}}\in \B_j$ and only retain the wave packets according to the parameters $H_{i,j+1}$, $M_{i,j+1}$ and $\bth_{i,j+1}$. The result is the function $F_i^{(j+2)}$.
	
	Each $B_{R_{j+1}}\in\B_j$ will intersect 
	\begin{equation}
		\label{cjhhupreyu8gu8tg8tg87g8}
		\sim N_{i,j}=r_{i,j+1}r_{i,j\subset j+1}
	\end{equation}
	special $T_j\in \T'_{i,j}(B_{R_j}^*)$, with directions structured according to the parameters $r_{i,j+1}$, $r_{i,j\subset j+1}$.
	
	We have for each $j\ge 0$
	\begin{equation}
		\label{rh uhruihuuihugyugyy}
		\lambda_{j-1}(B_{R_j}^*)\lessapprox \lambda_j(B_{R_{j+1}}^*)\times L_j
	\end{equation}
	where 
	
	\begin{equation}
		\label{rh uhruihuuihugyugyy1}
		L_j=
	\end{equation}
	$$\frac{\sum_{B_{R_{j+1}}\in\B_j}{(\sum_{\theta_{j+1}\in\Theta_{1,j+1}(B_{R_{j+1}})}\|F^{(j+2)}_{1,\theta_{j+1}}\|^2_{L^4(B_{R_{j+1}})})^2(\sum_{\theta_{j+1}\in\Theta_{2,j+1}(B_{R_{j+1}})}\|F^{(j+2)}_{2,\theta_{j+1}}\|^2_{L^4(B_{R_{j+1}})})^2}}{(\sum_{\theta_j\in\Theta_{1,j}(B_{R_j}^*)}\|F^{(j+1)}_{1,\theta_j}\|^2_{L^4(B_{R_j}^*)})^2(\sum_{\theta_j\in\Theta_{2,j}(B_{R_j}^*)}\|F^{(j+1)}_{2,\theta_j}\|^2_{L^4(B_{R_j^*})})^2}$$
	and
	$$\lambda_j(B_{R_{j+1}}^*)= $$$$\frac{\int_{B^*_{R_{j+1}}}|F^{(j+2)}_1F^{(j+2)}_2|^4}{(\sum_{\theta_{j+1}\in\Theta_{1,j+1}(B^*_{R_{j+1}})}\|F^{(j+2)}_{1,\theta_{j+1}}\|^2_{L^4(B^*_{R_{j+1}})})^2(\sum_{\theta_{j+1}\in\Theta_{2,j+1}(B^*_{R_{j+1}})}\|F^{(j+2)}_{2,\theta_{j+1}}\|^2_{L^4(B^*_{R_{j+1}})})^2}.$$

	We record the essential properties. The following basic estimate is universal, it does not rely on regularity.
	\begin{lemma}[Basic estimate]For each $j\ge -1$
		\begin{equation}
			\label{hhggryfgorueyregyyug}
			\lambda_j(B_{R_{j+1}}^*)\lessapprox \frac{\bth_{1,j+1}\bth_{2,j+1}}{R_{j+1}^2}.
		\end{equation}
	\end{lemma}
	\begin{proof}
		This follows via interpolation between $L^4$ and $L^\infty$. The proof of a more refined estimate is presented in Section \ref{sec:3.1}.
		
	\end{proof}
	\begin{lemma}
		\label{manipropert}
		For each $j\ge 0$
		\begin{equation}
			\label{ee1}
			\frac{H_{i,j+1}}{H_{i, j}} \leq r_{i, j \subset j+1}, 
		\end{equation}
		\begin{equation}
			\label{ee2}
			(\frac{H_{i,j+1}}{H_{i, j}})^2 \lessapprox R_{j+1}^{1/2}\frac{r_{i, j \subset j+1}}{M_{i,j+1}}
		\end{equation}
		\begin{equation}
			\label{ee3}
			N_{i,j}\le \bth_{i,j},
		\end{equation}
		\begin{equation}
			\label{ee4}
			\bth_{i,j+1}\le r_{i,j+1},
		\end{equation}
		\begin{equation}
			\label{ee5}
			\prod_{j=0}^Mr_{i, j \subset j+1}\le N_{i,0}.
		\end{equation}
		\begin{equation}
			\label{ee6}
			r_{i,j}\lesssim R_j^{\alpha/2}
		\end{equation}\begin{equation}
			\label{ee7}
			\;r_{i,j\subset j+1}\lesssim R_j^{\alpha/4}.
		\end{equation}
	\end{lemma}
	\begin{proof}
		\eqref{ee1} follows from the triangle inequality on  $B_{R_{j+1}}^{*}$. \eqref{ee2} follows from $L^2$ orthogonality on  $B_{R_{j+1}}^{*}$. The derivation of these inequalities has been explained for $j=0$, see  \eqref{n re hvhturtgu56i96i0-67} and \eqref{n re hvhturtgu56i96i0-671}. The case $j\ge 1$ follows via an identical argument. 
		
		\eqref{ee6} and \eqref{ee7} are consequences of $\alpha$-regularity.

		\eqref{ee3} and \eqref{ee4} are trivial, by definition.
		
		To prove \eqref{ee5}, we apply \eqref{ee3} and \eqref{ee4} repeatedly in the form $r_{i, j+1}\ge N_{i,j+1}$ as follows
		\begin{align*}
			N_{i,0}&=r_{i,0\subset1}r_{i,1}      
			\\&\ge r_{i,0\subset1}N_{i,1}
			\\&= (r_{i,0\subset1}r_{i,1\subset2})r_{i,2}
			\\&\ge (r_{i,0\subset1}r_{i,1\subset2})N_{i,2}
			\\&= (r_{i,0\subset1}r_{i,1\subset2}r_{i,2\subset3})r_{i,3}
			\\&\ldots\ldots
			\\&\ge \prod_{j=0}^Mr_{i, j \subset j+1}.
		\end{align*}
	\end{proof}
	We now estimate the growth factor $L_j$ in \eqref
	{rh uhruihuuihugyugyy1}
	$$\frac{\sum_{B_{R_{j+1}}\in\B_j}{(\sum_{\theta_{j+1}\in\Theta_{1,j+1}(B_{R_{j+1}})}\|F^{(j+2)}_{1,\theta_{j+1}}\|^2_{L^4(B_{R_{j+1}})})^2(\sum_{\theta_{j+1}\in\Theta_{2,j+1}(B_{R_{j+1}})}\|F^{(j+2)}_{2,\theta_{j+1}}\|^2_{L^4(B_{R_{j+1}})})^2}}{(\sum_{\theta_j\in\Theta_{1,j}(B_{R_j}^*)}\|F^{(j+1)}_{1,\theta_j}\|^2_{L^4(B_{R_j}^*)})^2(\sum_{\theta_j\in\Theta_{2,j}(B_{R_j}^*)}\|F^{(j+1)}_{2,\theta_j}\|^2_{L^4(B_{R_j^*})})^2}.$$
	We write
	$$\bT_{i,j}=\#\T_{i,j}'(B_{R_j^*})$$
	and recall that 
	$$\bT_{i,j}=M_{i,j}\bth_{i,j}.$$

	\begin{lemma}For each $j\ge 0$
		\begin{equation}
			\label{kjfifvuhvuihrturtut}
			L_j\approx\frac{H_{1,j+1}^4H_{2,j+1}^4\bth_{1,j+1}^2\bth_{2,j+1}^2M_{1,j+1}M_{2,j+1}}{H_{1,j}^4H_{2,j}^4\bth_{1,j}\bth_{2,j} R_{j}^{3/2}}\frac{\#\B_j}{\bT_{1,j}\bT_{2,j}}.
		\end{equation}
		
	\end{lemma}
	\begin{proof}For the denominator of $L_j$  \eqref{rhughtgytygygu9gu90} implies (see e.g.  Exercise 2.7 (c) in \cite{demeter2020fourier})
		$${(\sum_{\theta_j\in\Theta_{1,j}(B_{R_j}^*)}\|F^{(j+1)}_{1,\theta_j}\|^2_{L^4(B_{R_j}^*)})^2(\sum_{\theta_j\in\Theta_{2,j}(B_{R_j}^*)}\|F^{(j+1)}_{2,\theta_j}\|^2_{L^4(B_{R_j^*})})^2}
		$$$$
		\approx H_{1,j}^4H_{2,j}^4\bth_{1,j}^2\bth_{2,j}^2 M_{1,j}M_{2,j} |T_{j}|^2
		$$
		$$=H_{1,j}^4H_{2,j}^4\bth_{1,j}\bth_{2,j} \bT_{1,j}\bT_{2,j}R_j^3.$$	
		For the numerator we similarly have
		\begin{align*}
			(\sum_{\theta_{j+1}\in\Theta_{i,j+1}(B_{R_{j+1}})}\|F^{(j+2)}_{i,\theta_{j+1}}\|^2_{L^4(B_{R_{j+1}})})^2&\approx H_{i,j+1}^4\bth_{i,j+1}^2M_{i,j+1}|T_{j+1}|\\&=H_{i,j+1}^4\bth_{i,j+1}^2M_{i,j+1}R_{j+1}^{3/2}.
		\end{align*}

	\end{proof}

	\subsection{The first estimate via bilinear Kakeya}
	\label{sec:3.1}
	The purpose of this subsection is two-fold. It serves as a warm-up for the more difficult computations in the next section, and it illustrates the limitation of the bilinear Kakeya inequality by reproving the basic estimate \eqref{hhggryfgorueyregyyug}.

	We recall the well-known bilinear Kakeya bound 
	
	\begin{equation}
		\label{bilKbd}
		\#\B_j\lesssim \frac{\bT_{1,j}\bT_{2,j}}{N_{1,j}N_{2,j}}. 
	\end{equation}
	This follows by recalling that for each $j\ge 0$ and $i\in\{1,2\}$, each $B_{R_{j+1}}\in\B_j$
	intersects at least $N_{i,j}$ of the tubes in $\T_{i,j}'(B_{R_j^*})$.
	
	Using this and \eqref{kjfifvuhvuihrturtut} we find
	$$L_j\lessapprox\frac{H_{1,j+1}^4H_{2,j+1}^4\bth_{1,j+1}^2\bth_{2,j+1}^2M_{1,j+1}M_{2,j+1}}{H_{1,j}^4H_{2,j}^4\bth_{1,j}\bth_{2,j} R_{j}^{3/2}}.$$
	We combine this with the upper bound \eqref{ee2} for $M_{i,j+1}$ to get
	\begin{align}
		L_{j} &\lessapprox  \frac{r_{1,j\subset j+1}r_{2,j\subset j+1}\bth_{1,j+1}^2\bth_{2,j+1}^2H_{1,j+1}^2H_{2,j+1}^2}{N_{1,j}N_{2,j}\bth_{1,j}\bth_{2,j} H_{1,j}^2H_{2,j}^2 R_j}\nonumber\\&\lessapprox \frac{r_{1,j\subset j+1}r_{2,j\subset j+1}r_{1,j+1}^2r_{2,j+1}^2H_{1,j+1}^2H_{2,j+1}^2}{N_{1,j}N_{2,j}\bth_{1,j}\bth_{2,j} H_{1,j}^2H_{2,j}^2 R_j},\,\,\text{by }\eqref{ee4}\nonumber\\&=\frac{r_{1,j+1}r_{2,j+1}H_{1,j+1}^2H_{2,j+1}^2}{\bth_{1,j}\bth_{2,j} H_{1,j}^2H_{2,j}^2 R_j}, \;\;\text{by }\eqref{cjhhupreyu8gu8tg8tg87g8}\label{55new} .
	\end{align}
	By applying \eqref{55new} for $0\le j\le M$ we conclude,  since $R^{\frac{1}{2^M}}\sim 1$, 
	\begin{align*}
		\lambda_{-1} &  \lessapprox (r_{1,1}\cdots r_{1,M+1})(r_{2,1}\cdots r_{2,M+1}) \left( \frac{H_{1,1}}{H_{1,0}}\cdots \frac{H_{1,M+1}}{H_{1,M}} \right)^2\left( \frac{H_{2,1}}{H_{2,0}}\cdots \frac{H_{2,M+1}}{H_{2,M}} \right)^2 \nonumber \\
		&  \times \frac{1}{(\bth_{1,0}\cdots\bth_{1,M})(\bth_{2,0}\cdots\bth_{2,M})} \frac{1}{R^{1+\frac12+\ldots+\frac1{2^M}}}
		\nonumber \\&\lesssim \prod_{j=0}^Mr_{1,j+1}\prod_{j=0}^Mr_{2,j+1}\prod_{j=0}^M(r_{1,j\subset j+1})^2\prod_{j=0}^M(r_{2,j\subset j+1})^2\frac{1}{\prod_{j=0}^M\bth_{1,j}\prod_{j=0}^M\bth_{2,j}}\frac1{R^2},\;\;\text{by }\eqref{ee1}\nonumber\\&=\prod_{j=0}^Mr_{1,j\subset j+1}\prod_{j=0}^Mr_{2,j\subset j+1}\frac{\prod_{j=0}^MN_{1,j}\prod_{j=0}^MN_{2,j}}{\prod_{j=0}^M\bth_{1,j}\prod_{j=0}^M\bth_{2,j}}\frac1{R^2}\nonumber\\
		& \lesssim \frac{N_{1,0}N_{2,0}}{R^{2}},\;\;\text{by }\eqref{ee3}\text{ and }\eqref{ee5}.
	\end{align*}

	\subsection{ Szemer\'{e}di--Trotter-type estimate}
	We establish the following non-concentration result. We point out that this is only valid for $j\ge 1$. No meaningful non-concentration can be derived from our assumptions when $j=0$.
	\begin{proposition}
		\label{Fuesst}
		Let $j\ge 1$ and $B=B_{R_j}\in\B_{j-1}$.  Let $1\le s\le R_j^{1/2}$.
		Fix $i\in\{1,2\}$ and fix a contributing direction $\theta_j\in\Theta_{i,j}(B)$. Let $T_s\subset 100B$ be some $sR_j^{1/2}\times R_{j}$-rectangle  with the same long direction as the $R_j^{1/2}\times R_j$-rectangles $T_j\in \T'(\theta_j,B)$. Then
		$$\#\{ T_j \in \T'(\theta_j,B): T_{j} \subset T_s \} \lessapprox s^{1-\alpha} \left( \frac{H_{i,{j-1}}}{H_{i,j}} \right)^2 R_j^\alpha.$$
	\end{proposition}
	\begin{proof}
		Recall that we have the function
		$$F_i^{(j+1)}(x)=\sum_{T_j\in\T'_{i,j}(B)}F^{(j)}_{i,B,T_j}(x),\;\;x\in B,$$
		with $$\|F^{(j)}_{i,B,T_j}\|_{\infty}\sim H_{i,j}$$
		$$\T'_{i,j}(B)=\cup_{\theta_j\in\Theta_{i,j}(B)}\T'(\theta_j,B),$$
		$$\#\T'_{i,j}(B)\sim \bT_{i,j}.$$

		We prove two estimates for the quantity
		\[ \int_{T_s} |\sum_{T_j\in\T'(\theta_j,B)}F^{(j)}_{i,B,T_j}|^2. \]
		On the one hand, we have due to spatial orthogonality
		\begin{equation}
			\label{ck jhvufhvuigu t8guuh9uh}
			\int_{T_s} |\sum_{T_j\in\T'(\theta_j,B)}F^{(j)}_{i,B,T_j}|^2 \approx 
			(\#\{ T_j \in \T'(\theta_j,B): T_j\subset T_s \}) |T_{j}| H_{i,j}^2.
		\end{equation}
		We now seek another estimate. Due to spatial orthogonality and the fact that $\T'(\theta_j,B)\subset\T(\theta_j,B)$ we have
		$$\int_{T_s}|\sum_{T_j\in\T'(\theta_j,B)}F^{(j)}_{i,B,T_j}|^2\lesssim \int_{T_s}|\sum_{T_j\in\T(\theta_j,B)}F^{(j)}_{i,B,T_j}|^2.$$
		We recall the two-scale equivalence
		$$\sum_{T_j\in\T(\theta_j,B)}F^{(j)}_{i,B,T_j}\approx F^{(j)}_{i,B,\theta_{j}}=\sum_{\theta_{j-1}\subset\theta_j}F^{(j)}_{i,B,\theta_{j-1}}\text{ on }B.$$
		While the functions $F^{(j)}_{i,B,\theta_{j-1}}$ are orthogonal on $B$, this property does not persist on the smaller $T_s$. Instead, we split $\theta_j$ into boxes $\theta_{j,s}$ of length $s^{-1}R_{j}^{-1/2}$, and note that the functions 
		$$\sum_{\theta_{j-1}\subset \theta_{j,s}}  F_{i,B,\theta_{j-1}}^{(j)}$$
		are pairwise orthogonal on $T_s$ for distinct $\theta_{j,s}$,

		\[ \int_{T_s}|\sum_{\theta_{j-1}\subset\theta_j}F^{(j)}_{i,B,\theta_{j-1}}
		|^2\lessapprox \sum_{\theta_{j,s}\subset \theta_j}\int_{T_s} |\sum_{\theta_{j-1}\subset \theta_{j,s}}  F^{(j)}_{i,B,\theta_{j-1}}|^2 .
		\] 
		For each $\theta_{j-1}$, $F^{(j)}_{i,B,\theta_{j-1}}$ contributes with only one wave packet $F^{(j)}_{i,B,T_{j-1}}$ on $B$, and its magnitude is $\sim H_{i,j-1}$. We note that by the triangle inequality and $\alpha$-dimensionality we have
		\[ 
		|\sum_{\theta_{j-1}\subset \theta_{j,s}}  F^{(j)}_{i,B,\theta_{j-1}}| \lesssim 
		H_{i,{j-1}} \left( \frac{s^{-1}R_j^{-1/2}}{R_j^{-1}} \right)^{\alpha},
		\]
		and 
		\[ 
		\#\{ \theta_{j,s} :\, \theta_{j,s} \subset \theta_j \} \lesssim \left( \frac{R_j^{-1/2}}{s^{-1}R_j^{-1/2}}\right)^\alpha=s^\alpha.
		\]
		These lead to the second estimate 
		\begin{equation}
			\label{ck jhvufhvuigu t8guuh9uh1}
			\int_{T_s}|\sum_{T_j\in\T'(\theta_j,B)}F^{(j)}_{i,B,T_j}|^2\lessapprox |T_s|H_{i,{j-1}}^2\left( \frac{s^{-1}R_j^{-1/2}}{R_j^{-1}} \right)^{2\alpha}s^\alpha.
		\end{equation}
		Combining \eqref{ck jhvufhvuigu t8guuh9uh} and \eqref{ck jhvufhvuigu t8guuh9uh1} we find
		\[
		\#\{ T_j \in \T'(\theta_j,B): T_{j} \subset T_s \} \lessapprox s^{1-\alpha} \left( \frac{H_{i,{j-1}}}{H_{i,j}} \right)^2 R_j^\alpha.
		\]
	\end{proof}
	\begin{proposition}If $j\ge 1$ and $\alpha\le \frac12$
		\begin{equation}\label{53}
			\#\B_j \lessapprox 
			\frac{(\bT_{1,j}\bT_{2,j})^{1/2} \left( \frac{H_{1,{j-1}}}{H_{1,j}} \right)^2\left( \frac{H_{2,{j-1}}}{H_{2,j}} \right)^2 R_j^{2\alpha+\frac12} }{(N_{1,j}N_{2,j})^{3/2}}.
		\end{equation}
		If $\alpha>\frac12$
		\begin{equation}
			\label{ST-large-alpha}\#\B_j \lessapprox 
			\frac{(\bT_{1,j}\bT_{2,j})^{1/2} \left( \frac{H_{1,{j-1}}}{H_{1,j}} \right)^{1/\alpha}\left( \frac{H_{2,{j-1}}}{H_{2,j}} \right)^{1/\alpha} R_j^{3/2} }{(N_{1,j}N_{2,j})^{\frac{\alpha+1}{2\alpha}}}.
		\end{equation}
	\end{proposition}
	\begin{proof}For \eqref{53} we use Theorem 1.2 in \cite{DemWang}. The collection $\T'_{i,j}(B)$ may be rescaled by $R_j$ to consist of $\delta\times 1$ rectangles, with $\delta=R_j^{-1/2}$. The directions $\theta_j$ form a $(\delta,\alpha, C_{AD})$-Katz-Tao set, where $C_{AD}$ is the AD-regularity  constant of our original direction set. Also, for each direction, the set of rectangles is a $(\delta,1-\alpha,K_2)$-Katz Tao set. Proposition \ref{Fuesst} implies that 
		$$K_2\lessapprox \left( \frac{H_{i,{j-1}}}{H_{i,j}} \right)^2 R_j^\alpha.$$   
		Recall that each $B_{R_{j+1}}$ intersects at least $N_{i,j}$ many $T_j\in\T'_{i,j}(B)$. Theorem 1.2 in \cite{DemWang} now implies that 
		$$\#\B_j \lessapprox \frac{(K_2)^2\bT_{i,j}R_j^{1/2}}{(N_{i,j})^3}\lessapprox
		\frac{\bT_{i,j} \left( \frac{H_{i,{j-1}}}{H_{i,j}} \right)^4 R_j^{2\alpha+\frac12} }{(N_{i,j})^{3}}.$$Finally, \eqref{53} follows by taking the geometric average for $i\in\{1,2\}$.
		
		\eqref{ST-large-alpha} follows similarly from the analogous result in \cite{doprod}.
		
	\end{proof}
	\section{Proof of Theorem \ref{mmmain3}, the case $\alpha\le \frac12$}
	\label{sec:4}
	Theorem \ref{mmmain3}, in its equivalent local form, will be a consequence of the following result.
	\begin{theorem}
		We have
		$$\lambda_{-1}(B_R)\lessapprox R^{\frac{4\alpha}{5}-2}.$$
	\end{theorem}
	
	\begin{proof}
		We keep track only of powers of $R$; the losses allowed by $\lessapprox$ are suppressed. Put $R_j=R^{q_j}$, where $q_j=2^{-j}$, and write
		$$
		\tau_j=\log_R(\bth_{1,j}\bth_{2,j}),\quad
		a_j=\log_R(r_{1,j}r_{2,j}),\quad
		c_j=\log_R(r_{1,j\subset j+1}r_{2,j\subset j+1}),
		$$
		$$
		n_j=\log_R(N_{1,j}N_{2,j}),\quad
		m_j=\log_R(M_{1,j}M_{2,j}),\quad
		h_j=\log_R\frac{H_{1,j+1}H_{2,j+1}}{H_{1,j}H_{2,j}}.
		$$
		Also write $b=\log_R(\lambda_{-1}(B_R))$. The basic estimate \eqref{hhggryfgorueyregyyug}, applied with $j=-1$, gives
		\begin{equation}
			\label{basic-exp-form}
			b\le \tau_0-2.
		\end{equation}
		We now derive a family of estimates by using $L_0,L_1,\ldots,L_k$, for $k\le M\sim \log\log R$. For $L_0$ we use the bilinear Kakeya bound \eqref{bilKbd}, while for $L_j$, $j\ge 1$, we use the Szemer\'{e}di--Trotter-type bound \eqref{53}. If $\log_R(L_j)=\ell_j$, then \eqref{kjfifvuhvuihrturtut} gives
		\begin{equation}
			\label{Lzero-exp-bound}
			\ell_0\le 4h_0+2\tau_1+m_1-\tau_0-n_0-\frac32,
		\end{equation}
		and, for $j\ge 1$,
		\begin{equation}
			\label{Lj-exp-bound}
			\ell_j\le 4h_j-2h_{j-1}+2\tau_{j+1}+m_{j+1}-\frac32\tau_j-\frac12m_j-\frac32n_j+(2\alpha-1)q_j.
		\end{equation}
		Indeed, \eqref{Lzero-exp-bound} is obtained from
		$$
		\#\B_0\lesssim \frac{\bT_{1,0}\bT_{2,0}}{N_{1,0}N_{2,0}},
		$$
		and \eqref{Lj-exp-bound} follows from
		$$
		\#\B_j\lessapprox 
		\frac{(\bT_{1,j}\bT_{2,j})^{1/2} \left( \frac{H_{1,{j-1}}}{H_{1,j}} \right)^2\left( \frac{H_{2,{j-1}}}{H_{2,j}} \right)^2 R_j^{2\alpha+\frac12} }{(N_{1,j}N_{2,j})^{3/2}},
		$$
		together with $\bT_{i,j}=M_{i,j}\bth_{i,j}$.
		
		Iterating \eqref{rh uhruihuuihugyugyy} gives
		$$
		\lambda_{-1}(B_R)\lessapprox L_0L_1\cdots L_k\lambda_k(B^*_{R_{k+1}}).
		$$
		Combining this with \eqref{Lzero-exp-bound}, \eqref{Lj-exp-bound}, and \eqref{hhggryfgorueyregyyug}, we obtain, for each $k\ge 1$,
		\begin{align}
			\label{raw-all-Lj-exp}
			b &\le 2(1-2^{-k})\alpha-\frac52-\tau_0
			+\frac12\sum_{j=1}^{k}\tau_j+3\tau_{k+1}
			-n_0-\frac32\sum_{j=1}^{k}n_j \nonumber\\
			&\quad +\frac12\sum_{j=1}^{k}m_j+m_{k+1}
			+2\sum_{j=0}^{k-1}h_j+4h_k.
		\end{align}
		The  inequalities from Lemma \ref{manipropert} become, in exponent form, for $j\ge 0$
		\begin{equation}
			\label{struct-exp-ineq}
			n_j=a_{j+1}+c_j,\quad n_j\le \tau_j,\quad \tau_{j+1}\le a_{j+1},\quad h_j\le c_j,
		\end{equation}
		and
		\begin{equation}
			\label{MH-exp-ineq}
			m_{j+1}+2h_j\le q_{j+1}+c_j,\qquad a_j\le \alpha q_j.
		\end{equation}
		The first identity in \eqref{struct-exp-ineq} is just \eqref{cjhhupreyu8gu8tg8tg87g8}, while the last inequality in \eqref{MH-exp-ineq} follows from \eqref{ee6}.
		
		Fix $1\le K\le M$. We take the weighted arithmetic mean of \eqref{basic-exp-form} and \eqref{raw-all-Lj-exp} for $1\le k\le K$, with weights
		$$
		w_0=\frac13,\qquad w_k=\frac{4}{3^{k+1}}\quad (1\le k<K),\qquad w_K=\frac{2}{3^K}.
		$$
		These weights add up to $1$. The resulting weighted estimate gives
		\begin{equation}
			\label{weighted-all-Lj-exp}
			b\le \frac45(1-6^{-K})\alpha-2+V_K,
		\end{equation}
		where the exact remainder coming from the weighted sum is
		\begin{align}
			\label{VK-exact}
			V_K={}&-\frac13
			-\frac13\tau_0+\frac13\tau_1
			+\sum_{j=2}^{K}\frac{13}{3^j}\tau_j
			+\frac{2}{3^{K-1}}\tau_{K+1} \nonumber\\
			&-\frac23n_0-\sum_{j=1}^{K}\frac{1}{3^{j-1}}n_j \nonumber\\
			&+\frac13m_1+\sum_{j=2}^{K}\frac{5}{3^j}m_j
			+\frac{2}{3^K}m_{K+1} \nonumber\\
			&+\frac43h_0+\sum_{j=1}^{K-1}\frac{20}{3^{j+1}}h_j
			+\frac{8}{3^K}h_K.
		\end{align}
		We now explain how  a useful upper bound for $V_K$ follows from \eqref{struct-exp-ineq} and \eqref{MH-exp-ineq}. First use $\tau_0\ge n_0=a_1+c_0$ (in the form $-\frac13{\tau_0}-\frac23{n_0}\le -a_1-c_0$), $\tau_j\le a_j$ for $j\ge 1$, and $n_j=a_{j+1}+c_j$ for $1\le j\le K$. The $\tau,n$ terms in \eqref{VK-exact} then give
		\begin{align}
			\label{VK-after-taun}
			V_K\le{}&-\frac13
			-\frac23a_1
			+\sum_{j=2}^{K}\frac{4}{3^j}a_j
			+\frac{1}{3^{K-1}}a_{K+1}
			-c_0-\sum_{j=1}^{K}\frac{1}{3^{j-1}}c_j \nonumber\\
			&+\frac13m_1+\sum_{j=2}^{K}\frac{5}{3^j}m_j
			+\frac{2}{3^K}m_{K+1}
			+\frac43h_0+\sum_{j=1}^{K-1}\frac{20}{3^{j+1}}h_j
			+\frac{8}{3^K}h_K.
		\end{align}
		Next group the remaining $m,h$ terms according to the inequalities $m_{j+1}+2h_j\le q_{j+1}+c_j$ and $h_j\le c_j$:
		\begin{align}
			&\frac13m_1+\sum_{j=2}^{K}\frac{5}{3^j}m_j
			+\frac{2}{3^K}m_{K+1}
			+\frac43h_0+\sum_{j=1}^{K-1}\frac{20}{3^{j+1}}h_j
			+\frac{8}{3^K}h_K \nonumber\\
			={}&\frac13(m_1+2h_0)+\frac23h_0
			+\sum_{j=1}^{K-1}\left[\frac{5}{3^{j+1}}(m_{j+1}+2h_j)
			+\frac{10}{3^{j+1}}h_j\right] \nonumber\\
			&+\frac{2}{3^K}(m_{K+1}+2h_K)+\frac{4}{3^K}h_K \nonumber\\
			\le{}&\frac13(q_1+c_0)+\frac23c_0
			+\sum_{j=1}^{K-1}\left[\frac{5}{3^{j+1}}(q_{j+1}+c_j)
			+\frac{10}{3^{j+1}}c_j\right] \nonumber\\
			&+\frac{2}{3^K}(q_{K+1}+c_K)+\frac{4}{3^K}c_K \nonumber\\
			={}&\frac13+c_0+\sum_{j=1}^{K-1}\frac{5}{3^j}c_j
			+\frac{2}{3^{K-1}}c_K,
			\label{mh-bound}
		\end{align}
		where in the last line we used $q_j=2^{-j}$, hence
		$$
		\frac13q_1+\sum_{j=1}^{K-1}\frac{5}{3^{j+1}}q_{j+1}
		+\frac{2}{3^K}q_{K+1}=\frac13.
		$$
		Substituting \eqref{mh-bound} into \eqref{VK-after-taun}, the constants and the $c_0$ terms cancel, and we obtain
		\begin{equation}
			\label{VK-first-bound}
			V_K\le
			-\frac23a_1
			+\sum_{j=2}^{K}\frac{4}{3^j}a_j
			+\frac{1}{3^{K-1}}a_{K+1}
			+\sum_{j=1}^{K-1}\frac{2}{3^j}c_j
			+\frac{1}{3^{K-1}}c_K.
		\end{equation}
		Finally use $a_j\ge n_j=a_{j+1}+c_j$ (replacing $c_j$ with the larger $a_j-a_{j+1}$), successively for $1\le j\le K$. The terms telescope:
		\begin{align*}
			V_K
			&\le -\frac{2}{3^K}a_K+\frac{1}{3^{K-1}}(a_{K+1}+c_K)\\
			&= -\frac{2}{3^K}a_K+\frac{1}{3^{K-1}}n_K\\
			&\le \frac{1}{3^K}a_K
			\le \frac{\alpha}{6^K}.
		\end{align*}
		Inserting this into \eqref{weighted-all-Lj-exp} yields, for every $1\le K\le M$,
		\begin{equation}
			\label{cklh hfghrug i[rtuh[h9]]}
			b\le \left(\frac45+\frac{1}{5\cdot 6^K}\right)\alpha-2.
		\end{equation}
		We choose $K=M$. Since $R^{1/2^M}\sim 1$, the factor $R^{\alpha/(5\cdot 6^M)}$ is absorbed into the $\lessapprox$ notation. Hence
		$$
		\lambda_{-1}(B_R)\lessapprox R^{\frac45\alpha-2},
		$$
		as desired.
	\end{proof}

	\begin{remark}[Limit of the method]
		\label{kjhe foreygeug-98y9670=u90=679}
		The exponent $4\alpha/5$ is the best exponent that can be obtained from the inequalities used in this section alone. More precisely, the argument above only uses the basic estimate \eqref{basic-exp-form}, the estimates \eqref{raw-all-Lj-exp} coming from $L_0,L_1,\ldots,L_k$, and the structural inequalities \eqref{struct-exp-ineq}--\eqref{MH-exp-ineq}. These inequalities form an abstract linear programming problem in the exponents $b,\tau_j,a_j,c_j,n_j,m_j,h_j$. We now exhibit a feasible point for this linear programming problem for which all the available estimates give exactly
		$$
		b=\frac45\alpha-2.
		$$
		This shows that no manipulation of these same inequalities can force an exponent smaller than $4\alpha/5$.
		
		For $0\le j\le M$ and $1\le j\le M+1$, set
		$$
		b=\frac45\alpha-2,
		\qquad
		\tau_0=\frac45\alpha,
		$$
		$$
		\tau_j=a_j=\frac45\alpha q_j\quad (1\le j\le M+1),
		$$
		$$
		c_0=h_0=\frac25\alpha,
		\qquad
		c_j=h_j=\frac45\alpha q_{j+1}\quad (1\le j\le M),
		$$
		and
		$$
		m_{j+1}=\left(1-\frac45\alpha\right)q_{j+1}\quad (0\le j\le M),
		\qquad
		n_j=a_{j+1}+c_j\quad (0\le j\le M).
		$$
		Then
		$$
		n_j=\tau_j,
		\qquad
		\tau_{j+1}=a_{j+1},
		\qquad
		h_j=c_j,
		$$
		and
		$$
		m_{j+1}+2h_j=q_{j+1}+c_j.
		$$
		Also
		$$
		a_j=\frac45\alpha q_j\le \alpha q_j.
		$$
		Thus all the structural inequalities \eqref{struct-exp-ineq}--\eqref{MH-exp-ineq} are satisfied, with equality except for the harmless slack in $a_j\le \alpha q_j$.
		
		The basic estimate \eqref{basic-exp-form} is also sharp on this example, since
		$$
		\tau_0-2=\frac45\alpha-2=b.
		$$
		Moreover, substituting the same values into \eqref{raw-all-Lj-exp} gives, for every $1\le k\le M$,
		\begin{align*}
			&2(1-2^{-k})\alpha-\frac52-\tau_0
			+\frac12\sum_{j=1}^{k}\tau_j+3\tau_{k+1}
			-n_0-\frac32\sum_{j=1}^{k}n_j  \\
			&\qquad
			+\frac12\sum_{j=1}^{k}m_j+m_{k+1}
			+2\sum_{j=0}^{k-1}h_j+4h_k
			=\frac45\alpha-2=b.
		\end{align*}
		Therefore the basic estimate and every estimate obtained from $L_0,L_1,\ldots,L_k$ are simultaneously saturated at the exponent $4\alpha/5-2$ in the exponent bookkeeping.

	\end{remark}
	
	\section{Proof of Theorem \ref{mmmain3}, the case $\alpha>\frac12$}
	\label{sec:large-alpha-computation}
	This section records the computation obtained from the same iteration as in Section \ref{sec:4}, but using the Szemer\'{e}di--Trotter estimate \eqref{ST-large-alpha} in place of \eqref{53}. 
	\begin{theorem}
		Assume $\frac12<\alpha<1$. Then
		\[
		\lambda_{-1}(B_R)\lessapprox R^{\Phi(\alpha)-2},
		\]
		where
		\[
		\Phi(\alpha)=
		\begin{cases}
			\displaystyle \frac{2\alpha}{3-\alpha},& \displaystyle \frac12<\alpha\le \frac23,\\[6pt]
			\displaystyle \frac{-3\alpha^2+13\alpha-2}{4(3-\alpha)},& \displaystyle \frac23\le \alpha<1.
		\end{cases}
		\]
		The two formulae agree at $\alpha=\frac23$.
	\end{theorem}
	\begin{proof}
		We use the exponent notation from Section \ref{sec:4}: $q_j=2^{-j}$, $R_j=R^{q_j}$, and
		\[
		\tau_j=\log_R(\bth_{1,j}\bth_{2,j}),\quad
		a_j=\log_R(r_{1,j}r_{2,j}),\quad
		c_j=\log_R(r_{1,j\subset j+1}r_{2,j\subset j+1}),
		\]
		\[
		n_j=\log_R(N_{1,j}N_{2,j}),\quad
		m_j=\log_R(M_{1,j}M_{2,j}),\quad
		h_j=\log_R\frac{H_{1,j+1}H_{2,j+1}}{H_{1,j}H_{2,j}},
		\]
		and $b=\log_R(\lambda_{-1}(B_R))$. Put
		\[
		\beta=\frac{\alpha+1}{2\alpha},\qquad \delta=4-\frac1\alpha.
		\]
		The estimate for $L_0$ is still \eqref{Lzero-exp-bound}. For $j\ge1$, however, \eqref{kjfifvuhvuihrturtut} and \eqref{ST-large-alpha} give
		\begin{equation}
			\label{Lj-large-alpha-exp-bound}
			\ell_j\le 4h_j-\frac1\alpha h_{j-1}+2\tau_{j+1}+m_{j+1}
			-\frac32\tau_j-\frac12m_j-\beta n_j.
		\end{equation}
		Here the factor $R_j^{3/2}$ in \eqref{ST-large-alpha} cancels the factor $R_j^{-3/2}$ in \eqref{kjfifvuhvuihrturtut}. Iterating as in Section \ref{sec:4}, and then using the basic estimate at scale $R_{k+1}$, gives for each $k\ge1$
		\begin{align}
			\label{raw-large-alpha-exp}
			b &\le -\frac32-2^{-k}-\tau_0
			+\frac12\sum_{j=1}^{k}\tau_j+3\tau_{k+1}
			-n_0-\beta\sum_{j=1}^{k}n_j \nonumber\\
			&\quad +\frac12\sum_{j=1}^{k}m_j+m_{k+1}
			+\delta\sum_{j=0}^{k-1}h_j+4h_k.
		\end{align}
		We shall use \eqref{struct-exp-ineq} and \eqref{MH-exp-ineq}. We also record the consequence of \eqref{ee7}
		\begin{equation}
			\label{c-large-alpha-exp-bound}
			c_j\le \alpha q_{j+1}\qquad (j\ge0).
		\end{equation}
		In particular, $c_0\le \alpha/2$. Finally, from \eqref{struct-exp-ineq} we shall repeatedly use
		\begin{equation}
			\label{a-dominates-next-c}
			a_j\ge n_j=a_{j+1}+c_j\qquad (j\ge1).
		\end{equation}
		
		First assume $\frac12<\alpha\le\frac23$. Set
		\[
		\rho=\frac{1-\alpha}{2-\alpha},\qquad
		A=\frac{2(3\alpha-1)}{2-\alpha}.
		\]
		For $1\le K\le M$, take the weighted arithmetic mean of \eqref{basic-exp-form} and \eqref{raw-large-alpha-exp}, $1\le k\le K$, with weights
		\[
		w_0=\frac{2-3\alpha}{2-\alpha},\qquad
		w_k=\frac{2\alpha}{(2-\alpha)^2}\rho^{k-1}\quad (1\le k<K),\qquad
		w_K=\frac{2\alpha}{2-\alpha}\rho^{K-1}.
		\]
		These weights are non-negative and add up to $1$. Substituting $\tau_j\le a_j$, $n_j=a_{j+1}+c_j$, $h_j\le c_j$, and $m_{j+1}+2h_j\le q_{j+1}+c_j$ into the weighted estimate gives
		\begin{align}
			\label{large-alpha-case-one-rem}
			b\le{}& -2+\frac{2\alpha}{3-\alpha}
			-Aa_1+\frac{A}{2-\alpha}\sum_{j=2}^{K}\rho^{j-2}a_j
			+A\sum_{j=1}^{K-1}\rho^{j-1}c_j \nonumber\\
			&\quad +O_\alpha\big(\rho^K(a_{K+1}+c_K)\big).
		\end{align}
		The non-terminal terms in \eqref{large-alpha-case-one-rem} telescope. Indeed, since $1-\frac1{2-\alpha}=\rho$, \eqref{a-dominates-next-c} gives
		\[
		-Aa_1+Ac_1+\frac{A}{2-\alpha}a_2\le -A\rho a_2,
		\]
		and the same argument at the next scales gives
		\[
		-A\rho^{j-1}a_j+A\rho^{j-1}c_j+\frac{A}{2-\alpha}\rho^{j-1}a_{j+1}
		\le -A\rho^j a_{j+1}.
		\]
		Thus the non-terminal part of the remainder is non-positive. Since $a_K\le \alpha 2^{-K}$ and $c_K\le a_K$, the terminal contribution is $O_\alpha((\rho/2)^K)$. Therefore
		\[
		b\le -2+\frac{2\alpha}{3-\alpha}+O_\alpha((\rho/2)^K).
		\]
		Taking $K=M$ absorbs the final error into the $\lessapprox$ notation.
		
		It remains to consider $\frac23\le\alpha<1$. We keep the same $\rho=(1-\alpha)/(2-\alpha)$ and set
		\[
		E=\frac{(2\alpha-1)(3\alpha-1)}{\alpha(2-\alpha)}.
		\]
		For $2\le K\le M$, take the weighted arithmetic mean of \eqref{raw-large-alpha-exp}, $1\le k\le K$, with weights
		\[
		w_1=\frac{3(1-\alpha)}{2-\alpha},\qquad
		w_k=\frac{2\alpha-1}{(2-\alpha)^2}\rho^{k-2}\quad (2\le k<K),\qquad
		w_K=\frac{2\alpha-1}{2-\alpha}\rho^{K-2}.
		\]
		Here no weight is placed on the basic estimate. Again the weights are non-negative and add up to $1$. Applying the same substitutions as above gives
		\begin{align}
			\label{large-alpha-case-two-rem-before-c0}
			b\le{}& -\frac{5(2-\alpha)}{2(3-\alpha)}+\frac{3\alpha-2}{2\alpha}c_0
			-\frac32(a_1-a_2-c_1)-Ea_2 \nonumber\\
			&\quad +\frac{E}{2-\alpha}\sum_{j=3}^{K}\rho^{j-3}a_j
			+E\sum_{j=2}^{K-1}\rho^{j-2}c_j
			+O_\alpha\big(\rho^K(a_{K+1}+c_K)\big).
		\end{align}
		Now use \eqref{c-large-alpha-exp-bound} in the form $c_0\le \alpha/2$. The constant terms become
		\[
		-\frac{5(2-\alpha)}{2(3-\alpha)}+\frac{3\alpha-2}{4}
		=-2+\frac{-3\alpha^2+13\alpha-2}{4(3-\alpha)}.
		\]
		Also $a_1-a_2-c_1\ge0$ by \eqref{a-dominates-next-c}. The remaining non-terminal terms telescope exactly as before:
		\[
		-Ea_2+Ec_2+\frac{E}{2-\alpha}a_3\le -E\rho a_3,
		\]
		and, inductively,
		\[
		-E\rho^{j-2}a_j+E\rho^{j-2}c_j+\frac{E}{2-\alpha}\rho^{j-2}a_{j+1}
		\le -E\rho^{j-1}a_{j+1}.
		\]
		As in the previous case, the terminal contribution is $O_\alpha((\rho/2)^K)$. Hence
		\[
		b\le -2+\frac{-3\alpha^2+13\alpha-2}{4(3-\alpha)}+O_\alpha((\rho/2)^K).
		\]
		Taking $K=M$ completes the proof.
	\end{proof}
	
	\section{Proof of Theorem \ref{mmmain1}}
	\label{sec:5}
	
	\begin{definition}[Linear decoupling constant]
		Let $\Dec_{p,\alpha}(R)$ be the smallest constant such that
		for each  $(\alpha,R^{1/2},C_{AD})$-AD regular collection $\I=\I_{R^{1/2}}$, each  $F:\R^2\to\C$ with spectrum inside $\cN_{1/R}(\I)$ and each ball $B_R\subset \R^2$ of radius $R$ we have 
		$$
		\|F\|_{L^p(B_R)}\le \Dec_{p,\alpha}(R)(\sum_{\theta\in \Theta(\I)}\|F_\theta\|^2_{L^p(B_R)})^{1/2}.
		$$
	\end{definition}
	For complete rigor, $L^p(B_R)$ would need to be replaced with a weighted version. We omit such details. 
	Theorem \ref{mmmain1} may be equivalently reformulated as 
	$$\Dec_{p,\alpha}(R)\lessapprox_{C_{AD}}R^{\frac\alpha2(\frac12-\frac3p-c_{p,\alpha})}.$$
	We collect a few results.
	\begin{proposition}[Rescaling]
		\label{propo:parabooorescal}
		Consider an $(\alpha,R^{1/2},C_{AD})$-AD regular collection $\I=\I_{R^{1/2}}$ and some $F:\R^2\to\C$ with spectrum inside $\cN_{1/R}(\I)$. Let $\tau= \cN_{1/R}\cap (H\times \R)$ be the curved tube that projects down to an interval $H$ of length $\sim R^{-1/4}$, centered at some $c_H\in\cup_{I\in\I}I$. Then
		$$\|\sum_{\theta\subset \tau\atop{\theta\in\Theta(\I)}}F_{\theta}\|_{L^p(\R^2)}\lesssim \Dec_{p,\alpha}(R^{1/2})(\sum_{\theta\subset \tau\atop{\theta\in\Theta(\I)}}\|F_\theta\|^2_{L^p(\R^2)})^{1/2}.$$
	\end{proposition}
	\begin{proof}
		Let $L$ be the affine function mapping $H$ to $[-1,1]$.
		The collection of intervals $\I'=\{L(I):\;I\in\I,\;I\subset H\}$  is $(\alpha,R^{1/4},C_{AD})$-AD regular. The function
		$$(x,y)\mapsto(R^{1/4}(x-c_H),R^{1/2}(y-2c_Hx+c_H^2))$$
		maps $\tau$ to $\cN_{1/\sqrt{R}}(\I')$ and each $\theta\subset \tau$ to some $\theta'\in\Theta(\I')$.
		
	\end{proof}

	Let $J_1,J_2\subset [-1,1]$ be intervals separated by $\sim 1$.   For $i\in\{1,2\}$, let $\I_i$ be $(\alpha,R^{1/2},C_{AD})$-AD regular collections of intervals inside $J_i$. Call $\I_i'$ the $(\alpha,R^{1/4},C_{AD})$-AD regular collection of $R^{-1/4}$-intervals covering $\I_i$.
	
	We let as before  $\Theta_{i,0}$ and $\Theta_{i,1}$ denote $\Theta(\I_i)$ and $\Theta(\I'_i)$, respectively.

	Assume  $F_i:\R^2\to\C$ has spectrum inside $\cN_{1/R}(\I_i)$.
	Let $B$ be an arbitrary ball in $\R^2$, of radius bigger than 1. We use the notation 
	$$\|F_i\|_{p,R,B}=
	(\sum_{\theta\in \Theta_{i,0}}\|F_{i,\theta}\|^2_{L^{p}(B)})^{1/2},
	$$
	$$\|F_i\|_{p,\sqrt{R},B}=
	(\sum_{\tau\in \Theta_{i,1}}\|F_{i,\tau}\|^2_{L^{p}(B)})^{1/2}.
	$$
	\begin{definition}[Bilinear decoupling constant]
		Let $\BilDec_{p,\alpha}(R)$ be the smallest constant such that for each $F_1,F_2$ as above we have
		$$
		\|(F_1F_2)^{1/2}\|_{L^p(B_R)}\le \BilDec_{p,\alpha}(R)(\|F_1\|_{p,R,B_R}\|F_2\|_{p,R,B_R})^{1/2}.
		$$
	\end{definition}
	The proof of the next result is the same as in the case $\alpha=1$. It relies on the good rescaling properties of $AD$-regular collections, as described in the proof of Proposition \ref{propo:parabooorescal}.
	\begin{proposition}[Bilinear to linear reduction]
		\label{klen rhufy re8gi[rt9g[yth9670-]]}
		$$\Dec_{p,\alpha}(R)\lessapprox \BilDec_{p,\alpha}(R).$$
	\end{proposition}
	\medskip
	
	Let us now prove Theorem \ref{mmmain2} assuming Theorem \ref{mmmain3}. We restate Theorem \ref{mmmain2}  in its equivalent local formulation that will be needed later in our argument.

	\begin{theorem}[Improved local $(p,\frac{p}2)$ bilinear decoupling]
		\label{invprop1}
		Let $\alpha\in(0,1)$. 
		For each $4< p<\infty$ 
		there is $$\Gamma_{p,\alpha}<\alpha(\frac14-\frac1p)-\frac2p$$
		such that for each ball $B_R\subset \R^2$ of radius $R$
		\begin{equation}
			\label{inv2}
			\|(\prod_{i=1}^2|F_i|)^{1/2}\|_{L^p(B_R)}\lessapprox _{C_{AD}}R^{\Gamma_{p,\alpha}} (\prod_{i=1}^2\|F_i\|_{\frac{p}{2},R,B_R})^{\frac1{2}}.
		\end{equation}
	\end{theorem}
	\begin{proof}
		Let us first assume $p>8$.
		The proof of  \eqref{inv2} follows by
		interpolating \eqref{ ruehfr0eufui rt9go5tho0-uoj=-fghytjhipf urogut-gu}
		$$\|(\prod_{i=1}^2|F_i|)^{1/2}\|_{L^8(B_R)}\lessapprox_{C_{AD}} R^{\Gamma_{8,\alpha}}(\prod_{i=1}^2\|F_i\|_{4,R,B_R})^{\frac1{2}}
		$$
		with the consequence 
		$$\|(\prod_{i=1}^2|F_i|)^{1/2}\|_{L^\infty(B_R)}\lesssim R^{\alpha/4}(\prod_{i=1}^2\|F_i\|_{\infty,R,B_R})^{\frac1{2}}
		$$
		of the Cauchy--Schwarz inequality
		$$|F_i(x)|\le R^{\alpha/4}(\sum_{\theta\in\Theta_{i,0}}|F_{i,\theta}(x)|^2)^{1/2}.$$
		This interpolation is explained  in \cite{BD}, and it goes as follows. Via pigeonholing, there is a reduction to the case when each $F_i$ is a balanced $R$-function. This essentially means that
		$$\|F_i\|_{q,R,B_R}\sim \|F_i\|_{q_1,R,B_R}^{t}\|F_i\|_{q_2,R,B_R}^{1-t}$$ 
		whenever 
		$$\frac1q=\frac{t}{q_1}+\frac{1-t}{q_2}.$$
		Thus, writing $$\frac1p=\frac{t}{8}+\frac{1-t}{\infty},\;t=\frac{8}{p}$$
		and noting that
		$$\frac{t}{4}+\frac{1-t}{\infty}=\frac2p$$
		we find
		\begin{align*}
			\|(\prod_{i=1}^2|F_i|)^{1/2}\|_{L^p(B_R)}&\le \|(\prod_{i=1}^2|F_i|)^{1/2}\|^t_{L^8(B_R)}\|(\prod_{i=1}^2|F_i|)^{1/2}\|_{L^\infty(B_R)}^{1-t}\\&\lessapprox_{C_{AD}} R^{t\Gamma_{8,\alpha}+(1-t)\alpha/4}(\prod_{i=1}^2\|F_i\|_{4,R,B_R}^{t}\|F_i\|_{\infty,R,B_R}^{1-t})^{\frac1{2}}\\&\sim R^{\frac{8}{p}(\Gamma_{8,\alpha}-\frac\alpha4)+\frac\alpha4}(\prod_{i=1}^2\|F_i\|_{\frac{p}{2},R,B_R})^{\frac1{2}}.
		\end{align*}
		Since Theorem \ref{mmmain3} guarantees that $\Gamma_{8,\alpha}<\frac{\alpha}{8}-\frac14$, we find
		$$\Gamma_{p,\alpha}:=\frac{8}{p}(\Gamma_{8,\alpha}-\frac\alpha4)+\frac\alpha4<\alpha(\frac14-\frac1p)-\frac2p.$$
		When $4<p<8$ we similarly interpolate  \eqref{ ruehfr0eufui rt9go5tho0-uoj=-fghytjhipf urogut-gu} with 
		$$\|(\prod_{i=1}^2|F_i|)^{1/2}\|_{L^4(B_R)}\lesssim R^{-1/2}(\prod_{i=1}^2\|F_i\|_{2,R,B_R})^{\frac1{2}}.
		$$
		The latter is seen to be the classical bilinear restriction inequality. It does not admit an improved form under regularity assumptions.
		
	\end{proof}
	We also need the following square function version of \eqref{inv2}
	\begin{equation}
		\label{inv6}
		\|[\prod_{i=1}^2(\sum_{\theta\in \Theta_{i,0}}|F_{i,\theta}|^2)^{1/2}]^{1/2}\|_{L^p(B_R)}\lessapprox R^{-\frac{2}{p}}(\prod_{i=1}^2\|F_i\|_{\frac{p}{2},R,B_R})^{1/2},\;p\ge 4.
	\end{equation}
	This is not new. It was used in this exact form in \cite{BD}, to tackle the case $\alpha=1$.
	It is optimal for arbitrary collections of $\theta$, as can be seen by considering a collection of spatially overlapping wave packets, one for each direction.

	The proof of \eqref{inv6} follows by interpolating
	$$\|[\prod_{i=1}^2(\sum_{\theta\in \Theta_{i,0}}|F_{i,\theta}|^2)^{1/2}]^{1/2}\|_{L^\infty(B_R)}\le(\prod_{i=1}^2\|F_i\|_{\infty,R,B_R})^{\frac1{2}}
	$$
	with the classical square function estimate
	$$\|[\prod_{i=1}^2(\sum_{\theta\in \Theta_{i,0}}|F_{i,\theta}|^2)^{1/2}]^{1/2}\|_{L^4(B_R)}\lesssim R^{-\frac12}(\prod_{i=1}^2\|F_i\|_{2,R,B_R})^{\frac1{2}}.
	$$
	This estimate  is essentially equivalent to the bilinear Kakeya inequality.
	\bigskip

	Fix $6<p<\infty$. To simplify notation we also introduce the following quantities. First, define 
	\begin{equation}
		\label{m hfwe9 fur8-egurtgp h-67=j-7=}
		\xi=\frac{2}{p-2},\;\;\text{so that  }\;\frac2p=\frac{1-\xi}{p}+\frac{\xi}{2}
	\end{equation}
	and 
	$$\eta=\frac{p-6}{p(p-2)}.$$
	Note that
	\begin{equation}
		\label{kljufre  gu85ug90oy 0h 9u-j978}
		p>6\iff \xi<\frac12.
	\end{equation}
	For a fixed $0\le \beta\le 1$, let $A_\beta(R)$ be the smallest constant such that  
	\begin{equation}
		\label{inv999}
		\|(\prod_{i=1}^2|F_i|)^{1/2}\|_{L^p(B_R)}\le A_\beta(R)(\prod_{i=1}^2\|F_i\|
		_{p,R,B_R})^{\frac{1-\beta}2}(\prod_{i=1}^2\|F_i\|
		_{\frac{p}2,R,B_R})^{\frac\beta{2}},
	\end{equation}
	holds for arbitrary $F_i$ with spectrum inside $\cN_{1/R}(\I_i)$ and each $B_R$. 
	Note that
	$$A_0(R)=\BilDec_{p,\alpha}(R)$$
	and that $A_1(R)$ is the $(p,p/2)$ decoupling constant. In rough terms, we will upper bound $A_0(R)$ by $A_1(R)$, by making use of  the intermediate values $A_{\xi^s}(R)$.

	Note that by H\"older's inequality
	$$(\prod_{i=1}^2\|F_i\|
	_{\frac{p}{2},R,B_R})^{\frac1{2}}\lesssim (\prod_{i=1}^2\|F_i\|
	_{p,R,B_R})^{\frac1{2}}R^{\frac{2}{p}}.$$ Combining this with \eqref{inv999}  we get for each $\beta$
	\begin{equation}
		\label{7743048765898}
		\BilDec_{p,\alpha}(R)\lesssim A_{\beta}(R)R^{\frac{2\beta}{p}}.
	\end{equation}

	The following result is at the heart of the bootstrapping mechanism.
	\begin{proposition}
		\label{propinv5}
		(a) Inequality \eqref{inv999} holds true for $\beta=1$ with $$A_1(R)\lessapprox_{C_{AD}} R^{\Gamma_{p,\alpha}}.$$
		
		(b) Moreover,  for each $\beta\in (0,1]$
		$$A_{\beta\xi}(R)\lessapprox A_\beta(R^{1/2})(\Dec_{p,\alpha}(R^{1/2}))^{1-\xi\beta}R^{\beta\eta}.$$
	\end{proposition}
	\begin{proof}
		Note that (a) follows from \eqref{inv2}.
		
		Inequality (b) is exactly the same as in \cite{BD}. We recall the proof, for the reader's convenience. Consider a finitely overlapping cover of $B_R$ with balls $\Delta$ of radius $R^{1/2}$. Note that
		$$
		\|(\prod_{i=1}^2 F_i)^{1/2}\|_{L^p(B_R)}\lesssim (\sum_{\Delta}\|(\prod_{i=1}^2|F_i|)^{1/2}\|_{L^p(\Delta)}^p)^{1/p}.
		$$
		We use H\"older's inequality on  $\Delta$
		$$\|F_i\|_{\frac{p}{2},R^{1/2},\Delta}\le \|F_i\|_{p,R^{1/2},\Delta}^{1-\xi}\|F_i\|_{2,R^{1/2},\Delta}^{\xi}$$
		to get
		\begin{equation}
			\label{f. jgp rtugrt]ogh yijo78}
			\|(\prod_{i=1}^2|F_i|)^{1/2}\|_{L^p(\Delta)}
			\lesssim A_\beta(R^{1/2})\prod_{i=1}^2(\|F_i\|_{p,R^{1/2},\Delta}^{1-\xi\beta}\|F_i\|_{2,R^{1/2},\Delta}^{\xi\beta})^{\frac1{2}}.
		\end{equation}
		Write$$a_{\Delta,i}=\|F_i\|_{p,R^{\frac12},\Delta}$$
		and
		$$b_{\Delta}=(\prod_{i=1}^2\|F_i\|_{2,R^{\frac12},\Delta}^p)^{\frac1{2p}}.$$
		We have
		\begin{equation}
			\label{inv4}
			\|(\prod_{i=1}^2|F_i|)^{1/2}\|_{L^p(B_R)}
			\lesssim A_\beta(R^{1/2})(\sum_{\Delta}b_{\Delta}^{\xi\beta p}\prod_{i=1}^2a_{\Delta,i}^{\frac{1-\xi\beta}{2}p})^{1/p}.
		\end{equation}
		We process the sum using H\"older's inequality 
		\begin{equation}
			\label{harpottt1}
			\sum_{\Delta}b_{\Delta}^{\xi\beta p}\prod_{i=1}^2a_{\Delta,i}^{\frac{1-\xi\beta}{2}p}\le (\sum_{\Delta}b_{\Delta}^p)^{\xi\beta}\prod_{i=1}^2(\sum_{\Delta}a_{\Delta,i}^p)^{\frac{1-\xi\beta}{2}}.
		\end{equation}

		To sum the factors $a_{\Delta,i}^p$ we invoke first Minkowski's inequality then Proposition \ref{propo:parabooorescal} to get
		\begin{equation}
			\label{harpottt2}
			\sum_{\Delta}\|F_i\|_{p,R^{\frac12},\Delta}^p\lesssim\|F_i\|_{p,R^{\frac12},B_R}^p\lesssim \Dec_{p,\alpha}(R^{1/2})^{p}\|F_i\|_{p,R,B_R}^p.
		\end{equation}
		We next show how to sum the factors $b_{\Delta}^p$. 
		Note  that by $L^2$ orthogonality followed by H\"older's inequality, for each $\Delta$
		\begin{equation}
			\label{inv7}
			\|F_i\|_{2,R^{\frac12},\Delta}\lesssim \|F_i\|_{2,R,\Delta}\lesssim R^{\frac12-\frac1p}\|F_i\|_{p,R,\Delta}.
		\end{equation}
		
		Due to the uncertainty principle, each 
		$|F_{i,\theta}|$ is essentially constant on each $\Delta$. Thus, in particular we can write 
		$$\sum_{\Delta\subset B_R}\prod_{i=1}^{2}\|F_i\|_{p,R,\Delta}^{\frac{p}2}\lesssim \sum_{\Delta\subset B_R}\|\prod_{i=1}^2(\sum_{\theta\in\Theta_{i,0}}|F_{i,\theta}|^2)^{\frac1{4}}\|_{L^p({\Delta})}^p$$
		\begin{equation}
			\label{inv5}
			\lesssim\|\prod_{i=1}^2(\sum_{\theta\in\Theta_{i,0}}|F_{i,\theta}|^2)^{\frac1{4}}\|_{L^p({B_R})}^p.
		\end{equation}
		Now \eqref{inv7}, \eqref{inv5} and \eqref{inv6} lead to
		\begin{equation}
			\label{harpottt3}
			\sum_{\Delta}b_\Delta^p\lesssim R^{-2}R^{\frac{p}{2}-1}(\prod_{i=1}^2\|F_i\|_{\frac{p}{2},R,B_R})^{1/2}.
		\end{equation}
		To end the argument simply invoke the estimates \eqref{inv4}, \eqref{harpottt1}, \eqref{harpottt2} and \eqref{harpottt3}.
		
	\end{proof}
	\bigskip
	Write
	$$A_{\beta}(R)= R^{\psi(\beta)},$$
	and
	$$\Dec_{p,\alpha}(R)=R^{\gamma_{p,\alpha}}.$$
	\begin{proposition}
		\label{f mrtgirt hi09hi9yj0-9j-}
		Assume $p>6$. Then
		\begin{equation}
			\label{ssxkj sdyvorefrto] go6]h}
			\gamma_{p,\alpha}\le \frac{\psi(1)(1-2\xi)+2\eta}{1-\xi}.
		\end{equation}
	\end{proposition}
	\begin{proof}
		Proposition \ref{propinv5} implies that for each $s\ge 0$
		
		\begin{equation}
			\label{inv13}
			\psi(\xi^{s+1})\le \frac12\psi(\xi^s)+\frac{\gamma_{p,\alpha}}2(1-\xi^{s+1})+\eta\xi^s.
		\end{equation}
		Iterating \eqref{inv13}  gives
		\begin{equation}
			\label{inv15}
			\psi(\xi^s)\le \frac1{2^s}\psi(1)+\gamma_{p,\alpha}(1-2^{-s})+2(\frac\eta\xi-\frac{\gamma_{p,\alpha}}2)\frac{2^{-s}-\xi^{s}}{\xi^{-1}-2}.
		\end{equation}
		When combined with \eqref{7743048765898} for $\beta=\xi^s$ \eqref{inv15} implies
		$$
		\BilDec_{p,\alpha}(R)\lesssim R^{\frac{2\xi^s}{p}}R^{\frac1{2^s}\psi(1)+\gamma_{p,\alpha}(1-2^{-s})+2(\frac\eta\xi-\frac{\gamma_{p,\alpha}}2)\frac{2^{-s}-\xi^{s}}{\xi^{-1}-2}}.
		$$
		Using Proposition \ref{klen rhufy re8gi[rt9g[yth9670-]]}, this further leads to
		$$\gamma_{p,\alpha}\le \frac{2\xi^s}{p}+\frac1{2^s}\psi(1)+\gamma_{p,\alpha}(1-2^{-s})+2(\frac\eta\xi-\frac{\gamma_{p,\alpha}}2)\frac{2^{-s}-\xi^{s}}{\xi^{-1}-2},$$
		or 
		$$\gamma_{p,\alpha}\frac{1-\xi-\xi(2\xi)^s}{1-2\xi}\le \frac{2(2\xi)^s}{p}+\psi(1)+\frac{2\eta}{1-2\xi}(1-(2\xi)^s).$$
		Recalling \eqref{kljufre  gu85ug90oy 0h 9u-j978}, we let $s\to\infty$ to finish the proof.
		
	\end{proof}
	We now conclude the proof of Theorem \ref{mmmain1}. 
	Note that it is equivalent with 
	$$\gamma_{p,\alpha}<\frac\alpha2(\frac12-\frac3p).$$
	We may take $$\psi(1)=\Gamma_{p,\alpha}<\alpha(\frac14-\frac1p)-\frac2p=\frac{\alpha(p-4)-8}{4p},$$ due to Theorem \ref{mmmain2}.
	Then, using \eqref{ssxkj sdyvorefrto] go6]h} and recalling that $1-2\xi=\frac{p-6}{p-2}$, $1-\xi=\frac{p-4}{p-2}$ and $\eta=\frac{p-6}{p(p-2)}$ we find
	\begin{align}
		\gamma_{p,\alpha}&\le \frac{\psi(1)(1-2\xi)+2\eta}{1-\xi}\nonumber\\&\label{jh fh rufhu0 rey98g9}=\frac{p-6}{p-4}(\frac2p+\Gamma_{p,\alpha})\\&<\frac\alpha2(\frac12-\frac3p)\nonumber.
	\end{align}

	\begin{remark}\label{weh -fureigrtihyj u=jpi8=[ko []}
		The bootstrapping mechanism we used relies critically on the fact that $2\xi<1$. The number 2 comes from  the two-scale paradigm $R\mapsto R^{1/2}$. The value of $\xi$ in \eqref{m hfwe9 fur8-egurtgp h-67=j-7=} is enforced by the use of the Lebesgue exponents $p/2$ and 2 in connection with the exponent $p$. The latter is the only exponent where orthogonality holds in the absence of curvature (unless a priori $L^p$ orthogonality is enforced, as in \cite{chang2022decoupling}; see also Remark \ref{finrem}). The role of $p/2$ was explained in Remark
		\ref{chyurt8g0-5696 j9- boy0h u-0j670j= 5u0-}. It seems possible that the use of a different exponent in place of $p/2$, possibly also depending on $\alpha$, would lead to a more efficient bootstrapping argument.
	\end{remark}
	\begin{remark}
		\label{ckl jhuicrhii rt[pgi]0-h96}
		If \eqref{ ruehfr0eufui rt9go5tho0-uoj=-fghytjh} holds with $\Gamma_{p,\alpha}=-\frac2p$ for some $p\le \frac{4}\alpha$, then  $\gamma_{p,\alpha}=0$, due to \eqref{jh fh rufhu0 rey98g9}.
	\end{remark}
	\begin{remark}
		\label{finrem}
		The computations in this section can be done for arithmetic Cantor sets, with $L^2$ orthogonality replaced by $L^{p/3}$ quasi-orthogonality. A variant of Proposition \ref{propinv5} will continue to hold, subject to the following changes. Part (a) remains the same. Part (b) remains the same, with the new values
		$$\xi=\frac12,\;\text{ so that }\;\frac{2}{p}=\frac{1-\xi}{p}+\frac\xi{p/3}$$
		and
		$$\eta=\frac{\alpha\kappa_{p/3}}{8},$$
		where $\kappa_{p/3}$ is the constant in \eqref{hhuihreufy8gu8tug89uh8-578} (corresponding to $p/3$, rather than $p$).
		The new choice of $\xi$ is informed by our use of the triple $(p,p/2,p/3)$, in place of $(p,p/2,2)$. The value of $\eta$ is the one that emerges from running the argument in Proposition \ref{propinv5}, with $L^{p/3}$ used in place of $L^2$. For example, \eqref{f. jgp rtugrt]ogh yijo78} now reads
		$$
		\|(\prod_{i=1}^2|F_i|)^{1/2}\|_{L^p(\Delta)}
		\lesssim A_\beta(R^{1/2})\prod_{i=1}^2(\|F_i\|_{p,R^{1/2},\Delta}^{1-\xi\beta}\|F_i\|_{p/3,R^{1/2},\Delta}^{\xi\beta})^{\frac1{2}},
		$$    
		while \eqref{inv7} becomes
		$$
		\|F_i\|_{p/3,R^{\frac12},\Delta}\lesssim R^{\frac{\alpha\kappa_{p/3}}{4}}\|F_i\|_{p/3,R,\Delta}\lesssim R^{\frac{\alpha\kappa_{p/3}}{4}}R^{\frac3p-\frac1p}\|F_i\|_{p,R,\Delta}.
		$$
		This easily leads to 
		$$\eta=\frac{\alpha\kappa_{p/3}}{8}.$$
		Next, iteration as in Proposition \ref{f mrtgirt hi09hi9yj0-9j-} now gives for each $s\ge 0$ 
		$$\gamma_{p,\alpha}(s+2)\le \frac{4}{p}+2\psi(1)+4s\eta=\frac{4}{p}+2\Gamma_{p,\alpha}+4s\eta.$$
		Letting $s\to\infty$ leads to the result proved in \cite{chang2022decoupling}
		$$\gamma_{p,\alpha}\le \frac{\alpha\kappa_{p/3}}{2}.$$
		On the other hand, if \eqref{ ruehfr0eufui rt9go5tho0-uoj=-fghytjh} is established with a $\Gamma_{p,\alpha}$ close enough to $-\frac2p$ so that 
		\begin{equation}
			\label{njhufhrufgurtgu8tug90uh095}
			\frac2p+\Gamma_{p,\alpha}<\frac{\alpha\kappa_{p/3}}{2},
		\end{equation}
		then $s=0$ gives the better bound
		$$\gamma_{p,\alpha}\le \frac2p+\Gamma_{p,\alpha}.$$
		Note however that this upper bound is worse than \eqref{jh fh rufhu0 rey98g9}. In that regard, the use of $(p,p/2,p/3)$ is less efficient than the use of $(p,p/2,2)$. 
	\end{remark}

	\bibliographystyle{alpha}
	\bibliography{references}
	
\end{document}